\documentclass{article}

\usepackage[english]{babel}

\usepackage[letterpaper,top=2cm,bottom=2cm,left=3cm,right=3cm,marginparwidth=1.75cm]{geometry}

\usepackage[applemac]{inputenc} 		
\usepackage[T1]{fontenc}    		
\usepackage[colorlinks]{hyperref}      
\usepackage{url}            			
\usepackage{amsthm}
\usepackage{booktabs}       		
\usepackage{amsfonts}       		
\usepackage{amsmath}
\usepackage{wrapfig}
\usepackage{nicefrac}       		
\usepackage{microtype}      		
\usepackage{lipsum}				
\usepackage[square,numbers]{natbib}
\usepackage{mathtools}
\usepackage{algorithm}
\usepackage{algorithmicx}
\usepackage{algpseudocode}
\usepackage{graphicx}
\usepackage[most]{tcolorbox}
\usepackage{multicol,multirow}
\usepackage{indentfirst,latexsym,bm}
\usepackage{amsmath}
\usepackage{subfigure}
\usepackage{amssymb}
\usepackage{xcolor}
\usepackage{comment}
\usepackage{enumitem}
\usepackage{bbm}
\usepackage{tikz}
\usepackage{mdframed}
\usepackage{nicematrix}
\usepackage{bbding}
\usepackage{pifont}
\usepackage{stmaryrd}

\usetikzlibrary{arrows.meta,positioning}
\usepackage{pgfplots}
\pgfplotsset{compat=newest}
\usepgfplotslibrary{fillbetween}
\usetikzlibrary{shapes,decorations}
\usetikzlibrary{fit}

\colorlet{color1}{blue}
\colorlet{color2}{red!50!black}

\definecolor{ivory}{RGB}{218,215,203}

\definecolor{cuhkp}{RGB}{98,56,105} 	
\definecolor{cuhkpl}{RGB}{152,24,147} 	
\definecolor{cuhkb}{RGB}{219,160,1} 	
\definecolor{cuhkbd}{RGB}{178,129,0} 	
\definecolor{cuhkr}{RGB}{88,35,155}  	
\definecolor{blackp}{RGB}{0,0,0} 
\definecolor{redp}{RGB}{255,0,0}
\definecolor{orangep}{RGB}{255,128,0}
\definecolor{brownp}{RGB}{128,77,0}
\definecolor{yellowp}{RGB}{255,230,0}
\definecolor{greenp}{RGB}{128,230,0}
\definecolor{bluep}{RGB}{0,128,255}
\definecolor{purplep}{RGB}{152,24,147}
\definecolor{pinkp}{RGB}{230,0,128}
\definecolor{lavender}{rgb}{0.9, 0.9, 0.98}

\usepackage{hyperref}[6.83]

\hypersetup{
    colorlinks=true, 
    linkcolor=blue!80!black,  
    citecolor=blue!80!black, 
    urlcolor=magenta 
}

\RequirePackage[capitalize,nameinlink]{cleveref}

\crefformat{equation}{\textup{#2(#1)#3}}
\crefrangeformat{equation}{\textup{#3(#1)#4--#5(#2)#6}}
\crefmultiformat{equation}{\textup{#2(#1)#3}}{ and \textup{#2(#1)#3}}
{, \textup{#2(#1)#3}}{, and \textup{#2(#1)#3}}
\crefrangemultiformat{equation}{\textup{#3(#1)#4--#5(#2)#6}}%
{ and \textup{#3(#1)#4--#5(#2)#6}}{, \textup{#3(#1)#4--#5(#2)#6}}{, and \textup{#3(#1)#4--#5(#2)#6}}

\Crefformat{equation}{#2Equation~\textup{(#1)}#3}
\Crefrangeformat{equation}{Equations~\textup{#3(#1)#4--#5(#2)#6}}
\Crefmultiformat{equation}{Equations~\textup{#2(#1)#3}}{ and \textup{#2(#1)#3}}
{, \textup{#2(#1)#3}}{, and \textup{#2(#1)#3}}
\Crefrangemultiformat{equation}{Equations~\textup{#3(#1)#4--#5(#2)#6}}%
{ and \textup{#3(#1)#4--#5(#2)#6}}{, \textup{#3(#1)#4--#5(#2)#6}}{, and \textup{#3(#1)#4--#5(#2)#6}}

\crefdefaultlabelformat{#2\textup{#1}#3}

\theoremstyle{plain}
\newtheorem{theorem}{Theorem}[section]
\newtheorem{lemma}[theorem]{Lemma}
\newtheorem{corollary}[theorem]{Corollary}
\newtheorem{proposition}[theorem]{Proposition}
\newtheorem{definition}[theorem]{Definition}
\newtheorem{remark}[theorem]{Remark}

\theoremstyle{definition}

\newtheorem{example}[theorem]{Example}
\newtheorem{assumption}[theorem]{Assumption}
\crefname{assumption}{Assumption}{Assumptions}
\Crefname{assumption}{Assumption}{Assumptions}

\theoremstyle{remark}

\newtheorem*{fact*}{Fact}

\DeclareMathOperator*{\argmax}{argmax}
\DeclareMathOperator*{\argmin}{argmin}

\newcommand{\intset}[1]{\llbracket #1 \rrbracket}

\newcommand{\cL}{\mathcal L}
\newcommand{\cC}{\mathcal C}
\newcommand{\R}{\mathbb{R}}
\newcommand{\N}{\mathbb{N}}
\newcommand{\Rd}{\mathbb{R}^d}
\newcommand{\Prob}{\mathbb{P}}
\newcommand{\Exp}{\mathbb{E}}

\newcommand{\dist}{\mathrm{dist}}
\newcommand{\crit}{\mathrm{crit}}
\newcommand{\sL}{{\sf L}}
\newcommand{\sG}{{\sf G}}
\newcommand{\sA}{{\sf A}}
\newcommand{\sB}{{\sf B}}
\newcommand{\sC}{{\sf C}}
\newcommand{\sD}{{\sf D}}

\newcommand{\cO}{\mathcal O}
\newcommand{\cA}{\mathcal A}
\newcommand{\cB}{\mathcal B}
\newcommand{\cG}{\mathcal G}
\newcommand{\cF}{\mathcal F}
\newcommand{\cE}{\mathcal E}

\newcommand{\cX}{\mathcal{X}}
\newcommand{\cY}{\mathcal{Y}}
\newcommand{\cZ}{\mathcal{Z}}
\newcommand{\cD}{\mathcal{D}_{h}}

\newcommand{\err}{\mathcal{E}}

\newcommand{\dom}[1]{\mathrm{dom}(#1)}
\newcommand{\idom}[1]{\mathrm{int\ dom}(#1)}
\newcommand{\diam}{\mathrm{diam}}
\newcommand{\iprod}[2]{\langle #1, #2 \rangle}

\title{Shuffling the Stochastic Mirror Descent via Dual Lipschitz Continuity and Kernel Conditioning}
\author{Junwen Qiu\thanks{Industrial Systems Engineering and Management, National University of Singapore
  \\ Email: \{\texttt{jwqiu@nus.edu.sg,leileimei@u.nus.edu,junyuz@nus.edu.sg}\}} \and Leilei Mei${}^*$ \and Junyu Zhang${}^*$ }
\begin{document}
\maketitle
\begin{abstract}
    The global Lipschitz smoothness condition underlies most convergence and complexity analyses via two key consequences: the descent lemma and the gradient Lipschitz continuity. How to study the performance of optimization algorithms in the absence of Lipschitz smoothness remains an active area. The relative smoothness framework from Bauschke--Bolte--Teboulle (2017) and Lu--Freund--Nesterov (2018) provides an extended descent lemma, ensuring convergence of Bregman-based proximal gradient methods and their vanilla stochastic counterparts. However, many widely used techniques (e.g., momentum schemes, random reshuffling, and variance reduction) additionally require the Lipschitz-type bound for gradient deviations, leaving their analysis under relative smoothness an open area. To resolve this issue, we introduce the dual kernel conditioning (DKC) regularity condition to regulate the local relative curvature of the kernel functions. Combined with the relative smoothness, DKC provides a dual Lipschitz continuity for gradients: even though the gradient mapping is not Lipschitz in the primal space, it preserves Lipschitz continuity in the dual space induced by a mirror map. We verify that DKC is widely satisfied by popular kernels and is closed under affine composition and conic combination. With these novel tools, we establish the first complexity bounds as well as the iterate convergence of random reshuffling mirror descent for constrained nonconvex relative smooth problems. 
\end{abstract}


\maketitle
\section{Introduction}
\label{sec:intro}
Consider a nonconvex problem $\min_{x\in\mathcal{Z}} f(x)$ over a closed and convex set $\mathcal{Z}\subseteq\Rd$. A standard assumption for both deterministic and stochastic methods is the \emph{global Lipschitz smoothness} of $f$ over $\mathcal{Z}$, which yields two key properties that facilitate the convergence analysis:
\begin{subequations}
   \begin{align}
       \label{eq:traditional descent}
    |f(y) -  f(x) - \langle \nabla f(x),y-x \rangle| &\leq  \frac{\sL}{2}\,\|x-y\|^2,\qquad \forall x,y\in\mathcal{Z},\\[2mm]
    \label{eq:traditional Lipschitz bound}
    \|\nabla f(x) - \nabla f(y)\| &\leq \sL\|x-y\|, \!\qquad\quad \forall x,y\in\mathcal{Z}.
   \end{align}
\end{subequations}
However, the Lipschitz smoothness under $\ell_2$-distance may not properly capture the problem structure in many applications. For example, many reinforcement learning algorithms suffer much worse dimension-dependence if $\ell_2$-distance is used (instead of KL-divergence) \cite{wang2017primal,jin2020efficiently}, and the D-optimal design problem \cite{lu2018relatively} even fail to satisfy this condition. More examples include the quadratic inverse problem \cite{bolte2018first}, the Poisson inverse problem \cite{gao2021convergence}, and log-portfolio optimization \cite{fernholz2002stochastic}, etc.

To overcome the lack of a descent lemma \eqref{eq:traditional descent}, the concept of \emph{relative smoothness} was proposed in \cite{bauschke2017descent,bolte2018first,lu2018relatively}, from which the \emph{extended descent lemma} follows:
\begin{equation}
\label{eq:extended descent}
\big|f(x) - f(y) - \iprod{\nabla f(y)}{x-y}\big| \leq \sL \cD(x,y),
\end{equation}
where $\cD$ is the Bregman divergence induced by some kernel $h$, providing a mirror descent update:
\begin{equation}
    \label{eq:majorization}
    x^{k+1}=\argmin_z \, f(x^k)+ \langle v^k,z-x^k\rangle + \frac{1}{\alpha_k}\, \cD(z,x^k),  
\end{equation} 
with $v^k = \nabla f(x^k)$ and $0< \alpha_k \leq  1/\sL$. As a majorization minimization scheme, a sufficient descent in terms of $\cD(x^{k+1},x^k)$ is ensured by \eqref{eq:extended descent}, and this 
enabled a wealth of work on relative smooth optimization, as well as their vanilla stochastic counterparts, see \cite{bauschke2017descent,bolte2018first,lu2018relatively,zhang2018convergence,davis2018stochastic,ding2023nonconvex,fatkhullin2024taming}.

However, merely the extended descent lemma may not be sufficient to move beyond vanilla stochastic approximation schemes. For many widely adopted  stochastic optimization techniques such as momentum \cite{liu2020improved}, random reshuffling \cite{gurbuzbalaban2021random,mishchenko2020random,nguyen2021unified}, and variance reduction \cite{johnson2013accelerating,cutkosky2020momentum}, a Lipschitz-type bound like \eqref{eq:traditional Lipschitz bound} is explicitly required for controlling stochastic error, which is not directly available under the relative smoothness framework. 
To bridge this gap, we demonstrate that if a function $f$ satisfies the relative smoothness condition with some constant $\sL$ and kernel $h$, then a \emph{dual Lipschitz continuity} holds for all $x,y\in\cX\subseteq\idom{h}$,
\begin{equation}\label{eq:dual Lipschitz} 
\|\nabla f(x) - \nabla f(y)\|  \leq \sL\,\sqrt{\kappa(\cX)}\cdot \rho_h(x,y),  
\end{equation} 
where $\rho_h(x,y):=\|\nabla h(x)-\nabla h(y)\|$ is a distance in the dual space induced by the mirror map $\nabla h$, and 
$$\kappa(\cX):=\;\max\, \left\{\frac{\lambda_{\max}(\nabla^2h(x))} {\lambda_{\min}(\nabla^2h(y))}: x,y\in\cX \right\}$$ is the condition number of $h$ over $\cX$. When $h(x)=\frac12\|x\|^2$, one has $\kappa(\cX)\equiv1$. But for non-quadratic kernels, $\kappa(\cX)$ may vary as $\cX$ changes. To resolve the issue of a varying constant $\kappa(\cX)$, we note that Lipschitz bound is typically applied to consecutive iterates $x^k$ and $x^{k+1}$ whose dual space distance  can be \emph{uniformly bounded} by any predetermined constant $\delta>0$, under suitable step-size.  Hence, we expect to uniformly control $\kappa(\cX)$ for any $\cX$ with $\delta$-\emph{bounded dual diameter}. In this paper, we formalize this intuition and introduce a dual kernel conditioning (DKC) regularity assumption. Combining DKC and \eqref{eq:dual Lipschitz}, we can say the gradient $\nabla f$ satisfies a \emph{uniform local Lipschitz continuity} w.r.t. the non-Euclidean distance $\rho_h$, which is already sufficient for the purpose of bounding errors in many stochastic optimization methods.


In fact, the DKC assumption is very general and versatile. We show that this assumption holds for a wide class of kernels, including Shannon entropy, Burg's entropy, Fermi-Dirac entropy, exponential kernel, and power kernels. We also prove that the kernel function class satisfying DKC is closed under nondegenerate affine compositions and compatible conic combinations.  


Finally, to elaborate the power of the dual Lipschitz bound \eqref{eq:dual Lipschitz} and DKC regularity, we design and analyze a random reshuffling mirror descent (RRMD) method for solving the nonconvex relatively smooth empirical risk minimization problem: 
\begin{equation}\label{eq:opt problem}
\min_{x\in \mathcal{Z}}\quad f(x) := \frac1n{\sum}_{i=1}^n\; f_i(x).
\end{equation}
We show that RRMD requires $\cO(\epsilon^{-1.5})$ to find some properly defined $\epsilon$-stationary point. 
Using the dual metric $\rho_h$, if in addition the objective function $f$ is definable, we are further able to prove the full-sequence convergence to a critical point of $f$ without necessitating $\mathcal{Z}=\dom{h}=\Rd$. This is in significant contrast to existing results on relatively smooth definable problems that all require the full domain assumption \cite{ahookhosh2021bregman,bolte2018first,latafat2022bregman,mukkamala2020convex}, which excludes many important kernels like Shannon entropy and Burg's entropy.


\subsection{Related works} Besides the related literature on relative smoothness that are already mentioned in the previous introduction, in this section, we review a few closely related works on stochastic mirror descent (SMD) and random reshuffling (RR). \vspace{0.2cm}

\noindent{\bf Random reshuffling. } RR is widely adopted technique to speed up stochastic gradient methods that replaces the unbiased with-replacement sampling scheme by a without-replacement scheme. Despite being biased, RR offers superior empirical speed up \cite{bottou2009curiously,bottou2012stochastic,goodfellow2016deep} to SGD without causing computational overhead. 
Under strong convexity, RR achieves an $\cO(\epsilon^{-0.5})$ sample complexity for finding some $\bar{x}$ s.t. $\Exp[\|\bar{x}-x^*\|^2]\leq \epsilon$ \cite{gurbuzbalaban2021random,mishchenko2020random}, significantly outperforming $\cO(\epsilon^{-1})$ complexity of SGD \cite{bottou2018optimization}. For nonconvex problems, RR requires $\cO(\epsilon^{-1.5})$ samples to find some $\bar{x}$ s.t.  $\Exp[\|\nabla f(\bar{x})\|^2]\leq\epsilon$ \cite{mishchenko2020random,nguyen2021unified,qiu2025new}, superior to the $\cO(\epsilon^{-2})$ complexity of SGD. Besides, the asymptotic rates of RR and the fast convergence under the \L ojasiewicz inequality have also been discussed in \cite{li2023convergence}. \vspace{0.2cm}

\noindent{\bf Stochastic mirror descent. } Taking $v^k = \nabla f_{i_k}(x^k)$ in \eqref{eq:majorization} with $i_k$ being independently and randomly selected from $\intset{n}$, we obtain the standard stochastic mirror descent method, which is a Bregman counterpart of SGD method. Early SMD, also called stochastic mirror-prox, emerged from the convex optimization \cite[etc.]{nemirovski2009robust,ghadimi2012optimal} 
under standard Lipschitz smoothness condition, 
and is recently extended to a more general relative smoothness setting \cite{nazin2019algorithms,d2021stochastic,hanzely2021fastest}. The SMD is also studied under the nonconvex relatively smooth setting considered in this paper \cite{bolte2018first,zhang2018convergence,davis2018stochastic,ding2023nonconvex,fatkhullin2024taming}, with a typical $\cO(\epsilon^{-2})$ sample complexity under suitable stationarity measure. In this paper, we show how random reshuffling can speed up SMD and achieve an $\cO(\epsilon^{-1.5})$ complexity.  

\subsection{Notation} 
By convention, $\|\cdot\|$ represents the $\ell_2$-norm in Euclidean space. Now let us fix the block partition of $x$, then for any $\cX\subseteq \Rd$, we denote $\cX_j$ its projection to the subspace of the $j$-th block. We use $\cC^2(\cZ)$ to denote the set of twice continuously differentiable functions defined on $\cZ$. For a matrix $A$, we use $\lambda_{\min}(A)$ and $\lambda_{\max}(A)$ to denote its minimum and maximum eigenvalues, respectively.

\section{Dual Lipschitz continuity and dual kernel conditioning}  
\subsection{Relative smoothness}
In this paper, we default the kernels to be associated with the feasible region $\mathcal{Z}$. Namely, we require the kernel $h\in\cC^2(\mathrm{int}(\cZ))$ to be strictly convex and is essentially smooth on its domain $\dom{h}$. By essential smoothness, the iterates \eqref{eq:majorization} always remain in the interior $\mathrm{int}(\mathcal{Z})$, but are allowed to converge to a boundary point when necessary.

\begin{definition}[Relative smoothness] \label[definition]{def:relative smooth}
\emph{We say a function $f$ is $\sL$-smooth relative to the kernel function $h$ for some constant $\sL>0$ if
$\sL  h \pm f$ are both convex on $\mathcal{Z}$.}
\end{definition}

If the function $f$ is also twice continuously differentiable, \cref{def:relative smooth} can also be restated as $-\sL \nabla^2h(x)\preceq \nabla^2 f(x)\preceq \sL \nabla^2h(x)$ for $\forall x\in \mathcal{Z}$, see \cite{lu2018relatively}. 
Denote
$\cD(x,y):=h(x) - h(y) - \langle \nabla h(y),x-y \rangle$ the Bregman divergence induced by the kernel $h$, then the extended descent lemma \eqref{eq:extended descent} holds. In particular, by setting $h(x) = \frac{1}{2}\|x\|^2$, the relative smoothness reduces to the standard Lipschitz smoothness and the extended descent lemma \eqref{eq:extended descent} reduces to the standard descent lemma \eqref{eq:traditional descent}.


\subsection{Dual kernel conditioning and dual Lipschitz continuity}  Denote $h^*$ the convex conjugate of $h$, it is known that $(\nabla h)^{-1}=\nabla h^*$. 
Then the mirror descent step \eqref{eq:majorization} allows an alternative form: $ x^{k+1} = \nabla h^*\left(\nabla h(x^k)-\alpha^kv^k\right)$ with $v^k = \nabla f(x^k)$. This corresponds to the original interpretation when mirror descent was first introduced \cite{nemirovsky1983problem}, where the descent step in fact takes place in the \emph{dual space} $\mathrm{Im}(\nabla h)$ induced by the \emph{mirror map} $\nabla h$, hoping that the dual space can better adapt to the problem geometry under properly selected kernel $h$, see \Cref{fig:mirror}.  

Next, as motivated in the introduction, we establish a Lipschitz-type bound to fulfill the needs for more intricate theoretical tools in stochastic algorithms. First, adopting this dual space perspective of mirror descent, we define
$\rho_h(x,y):=\|\nabla h(x)-\nabla h(y)\|,$ which  clearly induces a (non-Euclidean) distance because it is positive definite, symmetric, and satisfies triangle inequality. 

Observe that many kernels like Shannon entropy possess a separable structure, i.e., the variable $x$ allows a partition $x = (x_{[1]},x_{[2]},\ldots,x_{[m]})$ into $m$ blocks so that 
$h(x)=\sum_{j=1}^m h_{j}(x_{[j]})$,
where $x_{[j]}$ denotes the $j$-th block of $x\in\Rd$, and $h_j$ functions are often identical but not necessarily so. Then we introduce the local kernel condition number. 

\begin{figure}[h]
\centering
\begin{minipage}[b]{0.34\textwidth}
    \centering
    \includegraphics[width=\textwidth]{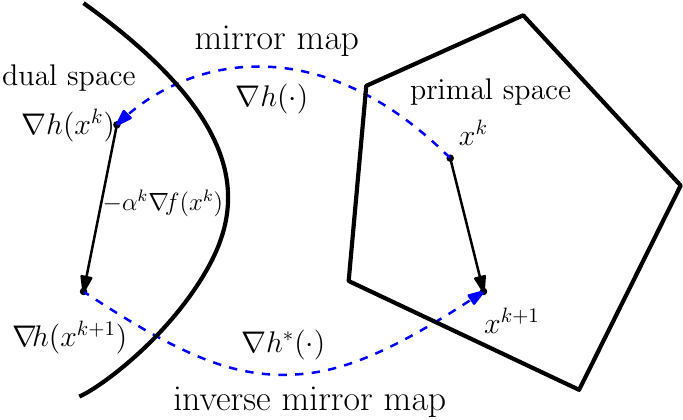} 
\end{minipage}
\hspace{-0.3cm}
\begin{minipage}[b]{0.42\textwidth}
    \centering
    \includegraphics[width=\textwidth]{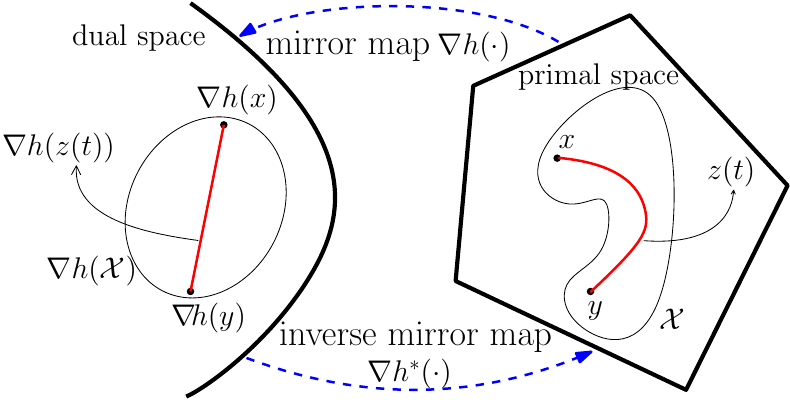}  
\end{minipage}
\hfill
\begin{minipage}[b]{0.22\textwidth}
    \centering
    \includegraphics[width=\textwidth]{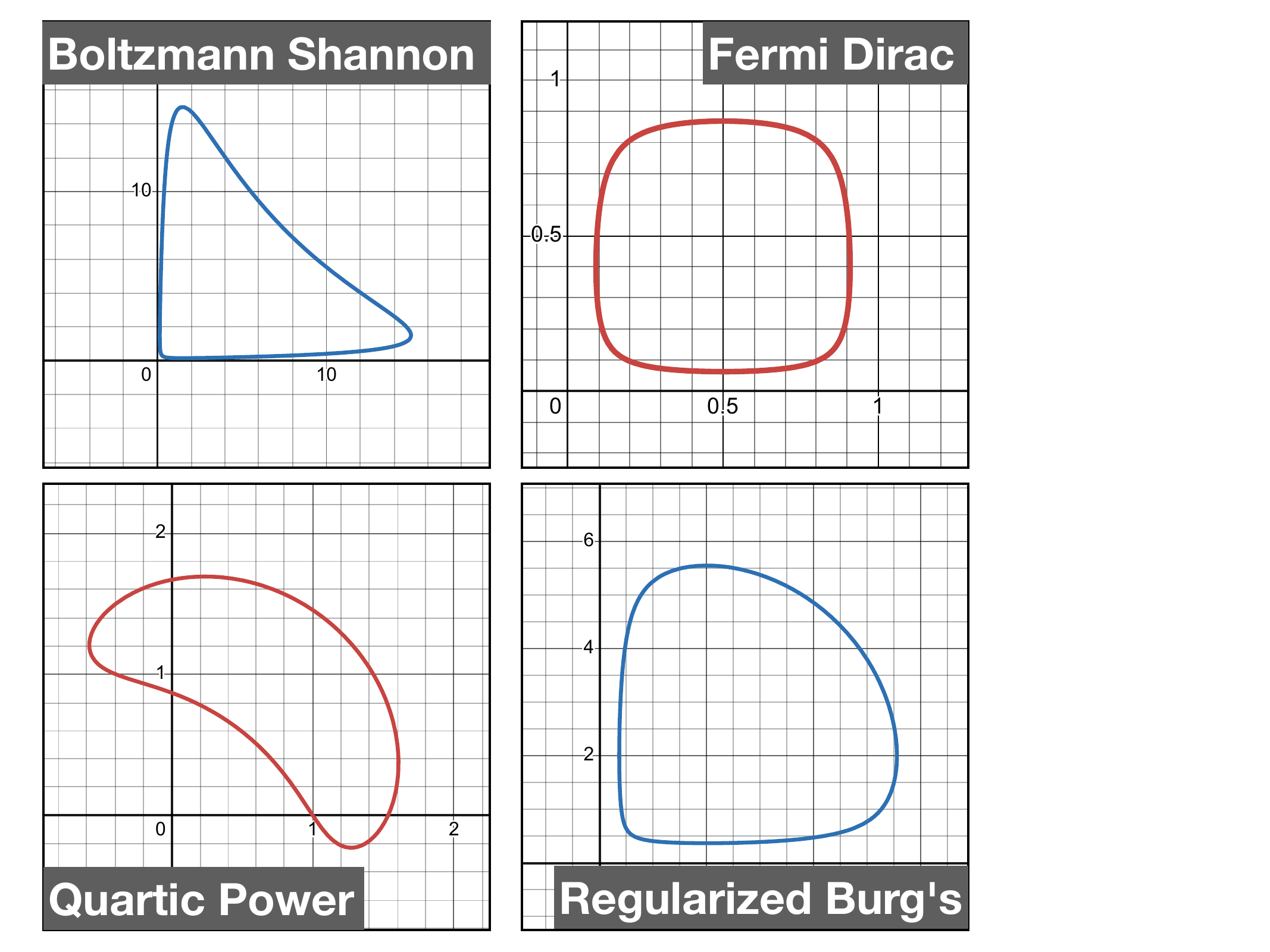}  
\end{minipage}
\caption{The first illustrates the dual space interpretation of mirror descent. The second illustrates the construction of the integration path $z(t)$ in \cref{lem:Lipschitz continuity,lem:important bounds}, by mapping the line segment between $\nabla h(x)$ and $\nabla h(y)$ back to the primal space. The third illustrates  $\nabla h^*(\mathcal{B})$ under quartic power kernel, where $\mathcal{B}$ is a ball in the dual space. This shows that when $\nabla h(\cX)$ is convex in dual space, $\cX$ may not necessarily be convex in primal space. } 
\label{fig:mirror}
\end{figure}
\vspace{-0.2cm}

\begin{definition}\label[definition]{def:mu and L}
    \emph{Let $h(x)=\sum_{j=1}^m h_{j}(x_{[j]})$ be a separable kernel associated with $\cZ$, define \[\mu_j(\cX_{j} ):={\min}_{x\in\cX_{j} } \lambda_{\min}\left(\nabla^2h_{j}(x)\right) \quad \text{and} \quad  \cL_j(\cX_{j} ):={\max}_{x\in\cX_{j} } \lambda_{\max}\left(\nabla^2h_{j}(x)\right),\]
    for any compact set $\cX\subseteq\idom{h}$. Then the local condition number of the kernel $h$ over $\cX$ is defined as $\kappa(\cX):=\max_{j\in\intset{m}}\big\{\kappa_j(\cX_j)\big\}$ with $\kappa_j(\cX_j):={\cL_j(\cX_{j} )}/{\mu_j(\cX_{j} )}$, and $\cX_j$ denotes the projection of $\cX$ to the subspace of the $j$-th block $x_{[j]}$, for each $j\in\intset{m}$.}
\end{definition}

By strict convexity and twice continuous differentiability of $h$, the quantity $\kappa(\cX)$ is always well-defined on compact $\cX$. Defining $\diam_{h}(\cY):=\sup\left\{\rho_{h}(x,y):x,y\in\cY\right\}$ the diameter of a set $\cY$ under the non-Euclidean distance $\rho_{h}(\cdot\,,\cdot)$, we introduce the dual kernel conditioning (DKC) regularity. 

\begin{definition}[DKC regularity]
\label[definition]{def:kernel cond}
\emph{We say an $m$-block kernel $h$ satisfies the DKC regularity if for any given $\delta>0$,   there exists a finite constant $\kappa_{\delta}>0$ such that $\kappa(\cX) \leq \kappa_{\delta}$, for all $\cX\in\Rd$ that satisfies $\diam_{h_j}(\cX_{j})\leq \delta$, for all  $\forall j\in\intset{m}$. } 
\end{definition}

That is, for any fixed $\delta>0$, every component $h_j$ is well-behaved with a \emph{uniformly} bounded condition number $\kappa_{\delta}$, as long as it is restricted in a local region $\cX_{j}$  with $\delta$-bounded $\rho_{h_j}$ diameter. This constant can often serve as a substitute for global Lipschitz constant in the analysis, enabling the analysis of random reshuffling mirror descent method (\cref{alg:shuffle}). Moreover, the concept of DKC is a natural extension of the kernel conditioning regularity proposed in \cite{zhang2024stochastic} that measures the size of local region $\cX_j$ by Euclidean distance and works exclusively for ``unconstrained'' kernels with $\dom{h} = \mathbb{R}^d$, excluding various important entropy-based kernels. By leveraging the novel non-Euclidean distance $\rho_{h_j}(\cdot\,,\cdot)$, DKC is naturally suitable for characterizing the local behavior of kernels with finite singularity, as illustrated in the following proposition, with verification details presented in Appendix \ref{app:kernel}.

\begin{proposition}\label[proposition]{prop kernel}
The following kernels all satisfy the DKC regularity condition:
\begin{itemize}[leftmargin=0.5cm, itemsep=0.00cm]
    \item Boltzmann-Shannon entropy $h(x):=\sum_{i=1}^d x_i\log(x_i)$ with $\dom{h}=\R_{++}^d$.
    \item Regularized Burg's entropy $ h(x) := \sum_{i=1}^d - \log(x_i) + \frac{\sigma}{2}\|x\|^2$, with $\sigma>0$ and  $\dom{h}=\R_{++}^d$.
    \item Fermi-Dirac entropy $ h(x):= \sum_{i=1}^d x_i\log(x_i) + (1-x_i) \log (1-x_i)$, with domain $\dom{h}=(0,1)^d$.\vspace{1mm}
    \item Power kernel: $h(x) := \frac{1}{r+2}\|x\|^{r+2} + \frac12\|x\|^2$, with $r\geq 0$ and domain $\dom{h}=\Rd$.
\end{itemize} 
\end{proposition}

As can be observed from the analysis of \cref{prop kernel}, the key motivation of introducing the DKC-regularity for separable kernels in a block-wise sense lies in the potentially inconsistent scaling among different variables. Consider the bivariate Boltzmann-Shannon entropy
$h(x)= \sum_{i=1}^2h_i(x_i)$ with  $h_i(x_i):= x_i\log x_i$, then $\nabla^2h(x) = \mathrm{diag}(x_1^{-1},x_2^{-1})$. Grouping $(x_1,x_2)$ as one whole block yields
\[\frac{\lambda_{\max}(\nabla^{2}h(x))}{\lambda_{\min}(\nabla^{2}h(x))} = \frac{\max\{x_1,x_2\}}{\min\{x_1,x_2\}} \to +\infty \quad \text{if}\quad  \frac{x_2}{x_1} \to +\infty.\]
That is, the kernel condition number can be huge even for singletons $\cX = \{x\}$ (hence $\diam_{h}(\cX)=0$), when $x_1$ and $x_2$ has largely different magnitudes. On the contrary, treating the kernel block-wise, \cref{prop kernel} demonstrates the well-behavedness of the kernel on each block.  Nevertheless, if the blocks do not diverge excessively, the blocks can still be stacked together as a larger block without encountering pathological condition numbers, which will be useful in the analysis of \cref{lem:Lipschitz continuity-m}.


Next, we introduce a desirable closedness property for the class of DKC-regular functions. For matrix $A$, denote $\sigma_{\min}(A)$ and $\sigma_{\max}(A)$ the minimum and maximum singular values, and denote $\kappa_A:=\frac{\sigma_{\max}(A)}{\sigma_{\min}(A)}$ the matrix condition number. The next proposition shows that the class of DKC-regular functions are well-posed and stay robust under various common operations.

\begin{proposition}[Closedness of DKC regularity]
\label[proposition]{theorem:closedness}
The DKC regularity is closed under scaling, compatible affine composition and conic combination in the following sense:\vspace{0.2cm}\\
\noindent{\emph{(a)}.} Let $h$ be a kernel with $m\geq1$ blocks, $A=\mathrm{Diag}(A_1,\cdots,A_m)\in\R^{d\times d}$ and $b\in\Rd$. If each $A_j$ is square and non-singular with size compatible to $x_{[j]}$, then the composition
$\phi(x):= h(Ax+b)$ satisfies DKC with
$\kappa_{\phi,\delta}\leq \max_{j\in\intset{m}} \{ \kappa_{A_j}^2\cdot\kappa_{h,{\delta}/{\sigma_{\min}(A_j)}} \}$.\vspace{0.25cm}\\
\noindent{\emph{(b)}.} Let $h$ and $g$ be two compatible kernels in the sense that $h$ and $g$ share the same block structure and $ \langle\nabla h(x)-\nabla h(y),\nabla g(x)-\nabla g(y)\rangle\geq0$. Then the conic combination $\phi(x):= \alpha h(x) + \beta g(y)$, $\alpha,\beta>0$ satisfies DKC with 
$\kappa_{\phi,\delta}\leq \max \left\{ \kappa_{h,{\delta}/{\alpha}},\kappa_{g,{\delta}/{\beta}}\right\}$.
\end{proposition}

It can be verified that all the kernels mentioned  in \Cref{prop kernel} satisfies the conditions of \Cref{theorem:closedness}, and are hence closed under the corresponding operations. \\

\noindent\emph{Proof of (a).}
As $h$ can be separated into $m$ blocks, it is then expressed as $
h(x) = \sum_{j=1}^m h_j(x_{[j]}).$ 
Noting that $A=\mathrm{Diag}(A_1,\ldots,A_m)$ with each $A_j$ being square and non-singular with size compatible to $x_{[j]}$, we can further write the composition $\phi(x):=h(Ax+b)$ as:
\[
\phi(x)={\sum}_{j=1}^m h_j(A_j x_{[j]} +b_{[j]}).
\]
We next focus on the $j$-th block. Define $\phi_j(z):=h_j(A_j z +b_{[j]})$, then 
\begin{align*}
\nabla \phi_j(z)&=A_j^\top \nabla h_j(A_j z + b_{[j]}),\quad \nabla^2 \phi_j(z)=A_j^\top \nabla^2 h_j(A_j z + b_{[j]}) A_j \\ \text{and} \quad \rho_{\phi_j}(y,z)&=\Big\|A_j^\top\big(\nabla h_j(A_j z + b_{[j]}) - \nabla h_j(A_j z + b_{[j]}) \big)\Big\|.    
\end{align*}

Let the set $\cX_j$, which is compatible to $x_{[j]}$, satisfy $\diam_{h_j}(\cX_j)\leq \delta$ under the metric $\rho_{\phi_j}(\cdot,\cdot)$. Thus, for all $y,z\in\cX_j$, we have 
\begin{align*}
    \delta \geq \rho_{\phi_j}(y,z) \geq \sigma_{\min}(A_j)\cdot \|\nabla h_j(A_j y + b_{[j]}) - \nabla h_j(A_j z + b_{[j]}) \|.
\end{align*}
Since $h_j$ fulfills DKC, based on the above relation, it holds that
\begin{equation}
    \label{eq:close DKC 1}
    \sup_{y,z\in\cX_j} \frac{\lambda_{\max}(\nabla^2 h_j(A_j y + b_{[j]}))}{\lambda_{\min}(\nabla^2 h_j(A_j z + b_{[j]}))}  \leq \kappa_{h,\delta}.
\end{equation}
Consider the condition number of $\phi_j$ on $\cX_j$. It readily follows from \eqref{eq:close DKC 1} that
\begin{align*}
    \sup_{y,z\in\cX_j} \frac{\lambda_{\max}(\nabla^2 \phi_j(y))}{\lambda_{\min}(\nabla^2 \phi_j(z))} &= \sup_{y,z\in\cX_j} \frac{\lambda_{\max}(A_j^\top \nabla^2 h_j(A_j y + b_{[j]}) A_j)}{\lambda_{\min}(A_j^\top \nabla^2 h_j(A_j z + b_{[j]}) A_j)} \\[1mm]
    &\leq \sup_{y,z\in\cX_j} \frac{\sigma_{\max}(A_j)\sigma_{\max}(A_j^\top) \cdot \lambda_{\max}( \nabla^2 h_j(A_j y + b_{[j]})) }{\sigma_{\min}(A_j)\sigma_{\min}(A_j^\top) \cdot\lambda_{\min}(\nabla^2 h_j(A_j z + b_{[j]}))} \leq \kappa_{A_j}^2 \cdot  \kappa_{h,\delta}.\qquad \Box
\end{align*}
\noindent\emph{Proof of (b).}
Due to the shared block structure of  $h$ and $g$, we can decompose $h,g$ and $\phi$ as
\[
h(x) = \sum_{j=1}^m h_{j}(x_{[j]}),\quad g(x) = \sum_{j=1}^m g_{j}(x_{[j]}) \quad \text{and} \quad \phi(x)=\sum_{j=1}^m \phi_j(x_{[j]}) = \sum_{j=1}^m \big(\alpha h_j(x_{[j]}) + \beta g_j(x_{[j]}) \big),
\]
where the integer $m \geq 1$. We analyze the $j$-th block, whose result can extend to all blocks.  Let the set $\cX_j$ satisfy $\diam_{h_j}(\cX_j) \leq \delta$ under the metric $\rho_{\phi_j}(\cdot,\cdot)$. It then follows for all $y,z\in\cX_j$ that
\begin{align*}
    \delta^2 &\geq \rho_{\phi_j}^2(y,z) = \|\nabla \phi_j(y) - \nabla \phi_j(z)\|^2\\
    &=\|\alpha (\nabla h_j(y)  - \nabla h_j(z)) + \beta (\nabla g_j(y) - \nabla g_j(z)) \|^2\\
    &=\alpha^2\|\nabla h_j(y)  - \nabla h_j(z)\|^2 + \beta^2 \|\nabla g_j(y) - \nabla g_j(z)\|^2 + 2\alpha\beta \langle \nabla h_j(y)  - \nabla h_j(z), \nabla g_j(y) - \nabla g_j(z)\rangle \\
    & \geq \alpha^2 \rho_{h_j}^2(y,z) + \beta^2 \rho_{g_j}^2(y,z),
\end{align*}
where the last line is due to $\langle \nabla h_j(y)  - \nabla h_j(z), \nabla g_j(y) - \nabla g_j(z)\rangle\geq 0$. Hence, we have $\rho_{h_j}(y,z)\leq \delta/\alpha$ and $\rho_{g_j}(y,z)\leq \delta/\beta$ for any $y,z\in\cX_j$. Recall that both $h$ and $g$ satisfy DKC, 
\[
\sup_{y,z\in\cX_j}\frac{\lambda_{\max}(\nabla^2 h_j(y))}{\lambda_{\min}(\nabla^2 h_j(z))}  \leq \kappa_{h,\delta/\alpha} \quad \text{and} \quad \sup_{y,z\in\cX_j}\frac{\lambda_{\max}(\nabla^2 g_j(y))}{\lambda_{\min}(\nabla^2 g_j(z))}  \leq \kappa_{g,\delta/\beta}.
\]
We now consider the condition number of $\phi_j$ on $\cX_j$, i.e.,
\begin{align*}
    \sup_{y,z\in\cX_j}\frac{\lambda_{\max}(\nabla^2 \phi_j(y))}{\lambda_{\min}(\nabla^2 \phi_j(z))} = \sup_{y,z\in\cX_j}\frac{\alpha \lambda_{\max}(\nabla^2 h_j(y)) + \beta \lambda_{\max}(\nabla^2 g_j(y)) }{\alpha \lambda_{\min}(\nabla^2 h_j(z)) + \beta \lambda_{\min}(\nabla^2 g_j(z))} \leq \max\{ \kappa_{h,\delta/\alpha}, \kappa_{g,\delta/\beta}\},
\end{align*}
where we utilize $\frac{a+b}{c+d} \leq \max\{\frac{a}{c},\frac{b}{d}\}$ to obtain the final inequality. \hfill$\Box$

\subsection{Dual Lipschitz continuity} Based on the previous preparation, we are able to provide the dual Lipschitz continuity bound under the $\rho_h$ distance in the dual space of mirror descent. 

\begin{lemma}\label[lemma]{lem:Lipschitz continuity}
Suppose $f$ is $\sL$-smooth relative to $h$ with $m=1$ block. Let $\cX\subseteq \idom{h}$ with image $\nabla h(\cX)$ being convex in the dual space. Then 
$\|\nabla f(x) - \nabla f(y)\| \leq \sL \sqrt{\kappa(\cX)}\cdot \rho_{h}(x,y)$, for $\forall x,y\in\cX$. 
\end{lemma} 

\begin{proof}
By default, we consider the single block case $m=1$. According to \cref{def:relative smooth}, we have  
\begin{equation} 
-\sL   I\preceq H_x^{-\frac12}\,\nabla^2 f(x)\,  H_x^{-\frac12} \preceq \sL  I\qquad\mbox{for}\qquad \forall x\in\mathcal{Z},\nonumber
\end{equation}
where $H_x:=\nabla^2 h(x)$. Then for any $x\in\mathcal{Z}$, we can observe that
\begin{equation}
    \label{eq:proof Lipschitz 1}
    \big\|\nabla^2 f(x)H_x^{-1} \big\| \leq \big\|H_x^{\frac12}\big\|\big\|H_x^{-\frac12}\nabla^2 f(x)H_x^{-\frac12}\big\|\big\|H_x^{-\frac12} \big\|\leq \sL\sqrt{\kappa(\cX)}.
\end{equation} 
Consider the parameterized curve $z(t):=\nabla h^*(\nabla h(x) + t\cdot u)$, $t\in[0,1]$, with $u:=\nabla h(y)-\nabla h(x)$. Clearly, $z(0) = x$ and $z(1) = y$. Due to the convexity of $\nabla h(\cX)$, the whole curve segment belong to $\cX$, see \Cref{fig:mirror}. Then using $\nabla f(y) -\nabla f(x) = \int_0^1 \nabla f(z(t))'\mathrm{d}t$, we have 
\begin{eqnarray}
    \|\nabla f(x) - \nabla f(y) \|  & = &  \Big\|\int_0^1\nabla^2 f(z(t))\nabla^2h^*(\nabla h(x) + t u)u\mathrm{d}t\Big\|\nonumber\\
    & \leq &  \int_0^1\!\!\big\|\nabla^2 f(z(t))\nabla^2h^*(\nabla h(z(t)))\big\|\mathrm{d}t \cdot\|u\|\nonumber
\end{eqnarray}
where the last inequality is due to the fact that $\nabla h(x)+tu = \nabla h(z(t))$. Using the fact that $\nabla^2 h^*(w)=[\nabla^2 h(x)]^{-1}$ with $x=\nabla h^*(w)$ and \eqref{eq:proof Lipschitz 1}, we further obtain that 
\begin{eqnarray*}
    \|\nabla f(x) - \nabla f(y) \| & \leq & \int_0^1 \big\|\nabla^2 f(z(t))[\nabla^2h(z(t))]^{-1}\big\|\mathrm{d}t \cdot\|u\| \\
    & \leq & \sL\sqrt{\kappa(\cX)}\cdot\|u\|.
\end{eqnarray*} 
Then substituting $\rho_h(x,y) = \|\nabla h(y)-\nabla h(x)\|$ to the above inequality proves the lemma.  
\end{proof}

As corollary of this lemma, if $f$ is $\sL$-smooth relative to a DKC-regular kernel $h$. Then fix any constant $\delta$ such that the corresponding $\kappa_\delta$ is mild, then for $\forall x,y\in\idom{h}$, we have 
\begin{equation}
    \label{eqn:Loc_Unif_Lip}
    \|\nabla f(x)-\nabla f(y)\|\leq \sL\sqrt{\kappa_\delta}\cdot\rho_h(x,y)\quad\mbox{if}\quad \rho_h(x,y)\leq \delta.
\end{equation} 
That is, although $\nabla f(\cdot)$ of a relatively smooth function may not be Lipschitz continuous under $\ell_2$-distance, it is locally Lipschitz continuous under a non-Euclidean $\rho_h$ distance with a uniformly bounded local Lipschitz constant. This gives us a better intuition of why the mirror descent can better adapt to the local geometry of the relative smooth optimization problems, providing an appropriate substitute of \eqref{eq:traditional Lipschitz bound} in a more general non-Euclidean setting. 
If $h$ has $m\geq2$ blocks, the situation is slightly more complicated, and is characterized by the following lemma.
\begin{lemma}\label[lemma]{lem:Lipschitz continuity-m}
Suppose $f$ is $\sL$-smooth relative to $h$ with $m\geq2$ blocks. Let $\cX=\cX_1\times\cdots\times\cX_m$, and $\cX_j\subseteq \idom{h_j}, j\in\intset{m}$. If both $\nabla h_j(\cX_j)$ and $\cX_j$ are convex for all $j\in\intset{m}$, then 
$$\sum_{j\in\intset{m}}\frac{\|\nabla_{[j]} \; f(x) - \nabla_{[j]} f(y)\|^2}{\mu_j(\cX_j)} \leq \sL^2 \Gamma^2(\cX) \kappa^2(\cX)\sum_{j\in\intset{m}} \frac{\rho_{h_j}^2(x_{[j]},y_{[j]})}{\mu_j(\cX_j)},\qquad \forall x,y\in\cX,$$
where $\Gamma(\cX) = \min  \Big\{8m, 8e\Big[2+3\ln\Big(\frac{\max_j\mu_j(\cX_j)}{\min_j\mu_j(\cX_j)}\Big)\Big]\Big\}^{1/2}$ is a finite constant.
\end{lemma} 
The proof is moved to Appendix \ref{proof:lem:Lipschitz continuity-m} for succinctness. This time, we assume both $\cX_j$ and $\nabla h_j(\cX_j)$ to be convex. It can be verified that all the separable kernels in \cref{prop kernel} satisfy this assumption. Finally, we end this section with a technical lemma that is used throughout the paper. 

\begin{lemma}\label[lemma]{lem:important bounds}
For $\forall j\in\intset{m}$, suppose $\cX_j\subseteq\idom{h_j}$ and the image $\nabla h_j(\cX_j)$ is convex, then 
\[
\frac{\rho_{h_j}(x_{[j]},y_{[j]})^2}{2\cL_j(\cX_j)}\leq \mathcal{D}_{h_j}(y_{[j]},x_{[j]}) \leq \frac{\rho_{h_j}(x_{[j]},y_{[j]})^2}{2\mu_j(\cX_j)}\quad \text{and} \quad \|x_{[j]}-y_{[j]}\| \leq \frac{\rho_{h_j}(x_{[j]},y_{[j]})}{\mu_j(\cX_j)}, \quad \forall x_{[j]},y_{[j]}\in\cX_j. 
\]
\end{lemma} 
\begin{proof}
In the proof, let us assume $h$ has $m=1$ block without loss of generality. If $m\geq2$, it suffices to repeat the same process for each $h_j$. Therefore, we will omit the subscript $[j]$ or $j$. Now, define  $g(t):=\cD(z(t),x)$, $t\in[0,1]$, with $z(t)$ being the same parameterized in the proof of \cref{lem:Lipschitz continuity}. Clearly, $g(0) = \cD(x,x) = 0$ and $g(1) = \cD(y,x)$. Then direct computation gives
\begin{equation}
\label{lem:important bounds 1}
\cD(y,x) = g(1)- g(0) = \int_0^1 g^\prime (t) \; \mathrm{d}t = \int_0^1 t \langle \nabla h(y) - \nabla h(x),\dot z(t) \rangle \;\mathrm{d}t. 
\end{equation}
Again, using $\nabla^2 h^*(w)=[\nabla^2 h(\nabla h^*(w))]^{-1}$ we get 
$\dot z(t)=[\nabla^2 h(z(t))]^{-1} (\nabla h(y)-\nabla h(x))$ and hence
\begin{equation}
    \label{lem:important bounds 2}
    \frac{\rho_h(x,y)^2}{\cL(\cX)}\leq \langle \nabla h(y) - \nabla h(x),\dot z(t) \rangle \leq \frac{\rho_h(x,y)^2}{\mu(\cX)}.
\end{equation} 
Combining \eqref{lem:important bounds 1} and \eqref{lem:important bounds 2} proves $\frac{\rho_h(x,y)^2}{2\cL(\cX)}\leq \cD(y,x) \leq \frac{\rho_h(x,y)^2}{2\mu(\cX)}$. Next, by $y-x = \int_0^1 \dot z(t) \mathrm{d}t$, we have 
\[
\qquad\qquad\qquad\,\|x-y\| \leq \int_0^1 \|\dot z(t)\| \; \mathrm{d}t \leq \rho_h(x,y) \cdot \int_0^1 \big\|[\nabla^2 h(z(t))]^{-1} \big\| \;\mathrm{d}t \leq \frac{\rho_h(x,y)}{\mu(\cX)},
\]
which completes the proof.
\end{proof}

\section{Algorithm and preliminary properties}
After introducing the tools of DKC-regularity and dual Lipschitz continuity, we illustrate how they can be used to analyze broader class of algorithms under relative smoothness by considering the 
random reshuffling mirror descent (RRMD, \cref{alg:shuffle}). In each epoch, RRMD generates a permutation $\pi^k$ of $\intset{n}$ and then performs mirror descent update \eqref{update} incrementally following the permuted data order. The permutation $\pi^k$ introduces a (biased) without-replacement sampling scheme to RRMD, and is the key difference to SMD that adopts an unbiased with-replacement sampling scheme, see \cite{nemirovski2009robust,ghadimi2012optimal}. 

In particular, when the permutations $\pi^k\equiv \pi^1$ remains identical for all $k$, we call \cref{alg:shuffle} the \emph{incremental mirror descent} method (IMD). This also covers the case where the data points are shuffled only once at the beginning of the algorithm. We say that the permutation $\pi^k$ is generated by \emph{uniform shuffling} scheme if $\{\pi^k_1,\ldots,\pi^k_n\}$ are sampled \emph{uniformly without replacement} from $\intset{n}$. 

\begin{algorithm}[h]
   \caption{Random Reshuffling Mirror Descent (RRMD)}
   \label{alg:shuffle}
\begin{algorithmic}
   \State {\bfseries Input:} kernel function $h:\Rd\to\R$, step-size $\{\alpha_k\}_k\subset\R_{++}$, initial point $x^1\in\Rd$
   \Repeat
   \State generate a permutation $\pi^k=\{\pi^k_1,\ldots,\pi^k_n\}$ of $\intset{n}:=\{1,\ldots,n\}$
   \State set $x^{k,1} = x^k$
   \For{$i=1$ {\bfseries to} $n$\vspace{-0.6cm}}
   \State 
   \begin{equation}
       \label{update}
            x^{k,i+1} \in\argmin_{u\in\mathcal{Z}} \left\{  \langle u, \nabla f_{\pi^k_i}(x^{k,i})\rangle + \frac{1}{\alpha_k} \cD(u,x^{k,i})\right\}
   \end{equation}
   \EndFor
   \State set $x^{k+1} = x^{k,n+1}$
   \Until{termination condition meets}
\end{algorithmic}
\end{algorithm}

By the dual space interpretation of mirror descent and the essential smoothness of $h$ on $\mathcal{Z}$, the iterations of \cref{alg:shuffle} will also  satisfy the following relationships:
\begin{equation}
\label{eqn:rrmd-update}
\nabla h(y^{k,i+1}) =\nabla h(y^{k,i}) - \alpha_k \nabla f_{\pi_i^k}(y^{k,i})\quad\mbox{and}\quad\nabla h(x^{k+1}) =\nabla h(x^{k}) - \alpha_k \sum_{i=1}^n \nabla f_{\pi_i^k}(y^{k,i}).
\end{equation}


\subsection{Assumptions}  
To simplify the notation and presentation, from now on, we will default the discussion to the single-block case $m=1$, which allows us to drop all the subscripts in the block notations $x_{[j]},\cX_j,\mu_j(\cdot),\cL_j(\cdot)$ as well as the frequently appearing summation over $j\in\intset{m}.$ This will greatly shorten expressions in the inequalities and will provide better readability and clarity, while all the proof can be parallelly extended to the general $m$-block kernel cases by directly adapting the proof to a block-by-block style. 
Next, we formally introduce a few basic assumptions. 
\begin{assumption}\label[assumption]{assumption 1}
There is a DKC-regular and $\mu$-strongly convex kernel $h$ associated with the feasible region $\mathcal{Z}$ and a constant $\sL>0$ so that each $f_i$ is $\sL$-smooth relative to $h$.
\end{assumption}

The relative smoothness condition (\Cref{assumption 1}) has become the standard framework for analyzing iterative methods that go beyond global Lipschitz continuity; see, e.g., \cite{ahookhosh2021bregman, bauschke2017descent, bolte2018first, latafat2022bregman, lu2018relatively}.  The validity of the DKC regularity has been extensively discussed. 

\begin{assumption}
    \label[assumption]{assumption 1.5} 
    There is an $f^*\in\R$ such that $f$ is lower bounded by $f^*$ on $\mathcal{Z}$. We also assume there exists a constant $\sG > 0$ such that 
    $\|\nabla f_i(x)\| \leq \sG$ for all $x\in\mathcal{Z}$ and $i\in \intset{n}$. 
\end{assumption}

We impose the bounded gradient assumption to clearly demonstrate how dual Lipschitz continuity and DKC regularity can be leveraged to establish complexity results and last-iterate convergence. \Cref{assumption 1.5}  helps isolate the key theoretical contributions from additional technical complications that arise when handling unbounded gradients. In \cref{subsec:relax bdd grad}, we also discuss how this assumption can be substantially relaxed. And the proof sketch is demonstrated therein.

\subsection{Preliminary properties}
Based on the preliminary discussions, we establish the elementary properties of RRMD. 
First, we introduce the set $\cX^k$ and the upper bound $\bar\alpha$ of the step-size:
\begin{equation}\label{def:Xk}
    \cX^k:=\big\{u\in\idom{h}: \rho_h(u,x^k) \leq \delta/2 \big\}\quad \text{and} \quad \bar\alpha:= \min\Big\{\frac{1}{2\sL\kappa_\delta},\frac{\delta}{2\sG}\Big\}.
\end{equation}
Therefore, using the update formula \eqref{eqn:rrmd-update} and \Cref{assumption 1.5}, it is clear that $0<\alpha_k \leq \bar \alpha/n$ indicates
\[\|\nabla h(y^{k,i+1}) - \nabla h(x^k)\| \leq \alpha_k   {\sum}_{j=1}^i \|\nabla f_{\pi^k_j}(y^{k,j})\| \leq \delta/(2n\sG) \cdot  i \sG \leq \delta/2, \qquad \forall i\leq n.\]
Then using \cref{assumption 1}, we 
obtain the following lemma that characterizes the update region of RRMD in each epoch and the condition number of the kernel $h$ on this region.

\begin{lemma}
\label[lemma]{lem:region}
Given \cref{assumption 1,assumption 1.5}, and let $\{x^k\},\{y^{k,i}\}$ be generated by \cref{alg:shuffle} with step-sizes $0<\alpha_k \leq \bar \alpha/n$. Then $y^{k,i}\in\cX^k$ for all $i\in\intset{n+1}$ and $k\in\N$ and $\kappa(\cX^k) \leq \kappa_\delta$.
\end{lemma}

That is, as long as $\alpha_k$ is bounded above by $\bar \alpha$, the entire epoch remains within a reasonable region $\cX^k$, on which the condition number of DKC-regular kernel $h$ is uniformly controlled by $\kappa_\delta$. Next, we present the preliminary descent of the algorithm.

\begin{lemma} 
\label[lemma]{lem:preliminary descent}
Under the same setting of \cref{lem:region}, the iterates generated by \cref{alg:shuffle}  satisfies
\begin{align*}
  f(x^{k+1}) &+ \frac{\cD(x^k,x^{k+1})}{2n\alpha_k} \leq f(x^k)  + \underbracket{\frac{\kappa_{\delta}\cdot\alpha_k}{2\mu(\cX^k)} {\sum}_{i=1}^n\| \nabla f_{\pi^k_i}(x^{k}) - \nabla f_{\pi^k_i}(y^{k,i}) \|^2}_{=:\err_k}.
\end{align*}  
\end{lemma}
The proof of this lemma can be found in Appendix \ref{proof:lem:preliminary descent}. It provides the preliminary descent-type property characterized by the Bregman difference $\cD(x^{k},x^{k+1})$. 

\subsection{Stationarity measure}
Different from the deterministic setting \cite{bauschke2017descent,bolte2018first}, $\cD(x^{k},x^{k+1})$ here does not directly relate to the gradient. Therefore, we introduce the following stationarity measure tailored for stochastic mirror descent methods.
\begin{definition}\label[definition]{def:stationary}
\emph{Let $\{x^k\}$ be generated by \cref{alg:shuffle} with step-size $\{\alpha_k\}_k$, we define
\begin{equation}
    \label{eq:station measure}
    \cG(x^k) := \frac{\cD(x^k,\hat x^k)}{(n\alpha_k)^2}\quad \text{where} \quad \hat x^k:=\argmin_{u\in\mathcal{Z}}\Big\{\langle \nabla f(x^k),u \rangle + \frac{\cD(u,x^k)}{n\alpha_k}\Big\}.
\end{equation}
}
\end{definition}
When  $h(x):=\frac12\|x\|^2$,  stationarity measure $\cG(x^k)$ reduces to  $\|\nabla f(x^k)\|^2$. 
Then, we establish a crucial relationship between $\cG(x^k)$ and $\cD(x^k,x^{k+1})$. The proof can be found in Appendix \ref{proof of lem:valid complexity}.
\begin{lemma}\label[lemma]{lem:valid complexity}
Under the setting of \cref{lem:region},  we have  $\hat{x}^k\in\cX^k$ and 
\[\frac{\|\nabla f(x^k)\|^2}{\mu(\cX^k)} \leq 2\kappa_\delta \cdot \cG(x^k)\quad \text{and} \quad 
\cG(x^k) \leq 2\kappa_\delta \cdot \frac{\cD(x^k, x^{k+1})}{(n\alpha_k)^{2}} + \frac{2\err_k}{n\alpha_k}.\]
\end{lemma}

\subsection{Error estimates and approximate descent property}
Let $(\Omega,\cF_k,\Prob)$ be the underlying probability space of the stochastic process $\{x^k\}$ generated by \cref{alg:shuffle}, where the filtration $\cF_k:=\sigma(x^1,x^2,\ldots,x^k)$ has a deterministic initial point $x^1\in\Rd$. With $\Exp_k[\cdot]:=\Exp[\cdot|\cF_k]$ denoting the conditional expectation, and the fact that $\mu(\cX^k)\geq\mu$ due to \cref{assumption 1}, we bound the stochastic error $\err_k$, with derivation in Appendix \ref{proof:lem:err bound}. 
\begin{proposition}[Stochastic error bound]
\label[proposition]{lem:err bound}
Consider the same setting of \cref{lem:region}.
\begin{enumerate}[label=\textup{\textrm{(\alph*)}},leftmargin=2.7em,topsep=2pt,itemsep=.5ex,partopsep=0ex]
 \item Then, for arbitrary permutation $\pi^k$, it holds deterministically that
$\err_k \leq \frac{(\kappa_\delta  \sL\sG)^2}{2\mu} \cdot n^3 \alpha_k^3.$
\item  Moreover, if $\pi^k$ is generated by uniform shuffling scheme, it holds in expectation that
\[
\Exp_k[\err_k]\leq 4\sL^2 (\kappa_\delta  n\alpha_k)^3\cdot \cG(x^k) + \frac{4(\kappa_\delta  \sL\sG)^2}{\mu} \cdot n^2\alpha_k^3.
\]
\end{enumerate}
\end{proposition}
Combining the previous results, we obtain the approximate descent property for RRMD.

\begin{proposition}[Approximate descent] \label[proposition]{prop:descent property}
Consider the same setting of \cref{lem:region}.
\begin{enumerate}[label=\textup{\textrm{(\alph*)}},leftmargin=2.7em,topsep=2pt,itemsep=.5ex,partopsep=0ex]
 \item Then, for arbitrary permutation $\pi^k$, it holds deterministically that
\begin{subequations}
    \begin{align}
        \label{eq:prop descent 0}
        f(x^{k+1}) + \frac{\cD(x^k,x^{k+1})}{2n\alpha_k} &\leq  f(x^k) + \frac{(\kappa_\delta  \sL\sG)^2}{2\mu} \cdot n^3 \alpha_k^3, \\[2mm]
        \label{eq:prop descent 1}
        f(x^{k+1}) + \frac{n\alpha_k}{4\kappa_\delta}\cdot \cG(x^k) &\leq f(x^k) + \frac{3(\kappa_\delta  \sL\sG)^2}{4\mu} \cdot n^3 \alpha_k^3.
    \end{align}
\end{subequations}
\item  Moreover, if $\pi^k$ is generated by uniform shuffling scheme, and we require in addition that $\alpha_k \leq 1/(7\kappa_\delta^{2} n \sL)$, then it holds in expectation that
\begin{equation}
\label{eq:prop descent 2}
\Exp[f(x^{k+1})] + \frac{n\alpha_k}{8\kappa_\delta} \cdot \Exp[\cG(x^k)]
\leq \Exp[f(x^k)] + \frac{6(\kappa_\delta  \sL\sG)^2}{\mu} \cdot n^2\alpha_k^3.
\end{equation} 
\end{enumerate} 
\end{proposition}
\begin{proof}    
By \Cref{lem:preliminary descent} and \Cref{lem:err bound} (a), we have
\begin{equation} 
     f(x^{k+1}) + \frac{\cD(x^k,x^{k+1})}{2n\alpha_k} \leq f(x^k) + \mathcal{E}_k \leq f(x^k) + \frac{(\kappa_\delta  \sL\sG)^2}{2\mu} \cdot n^3 \alpha_k^3, \nonumber
\end{equation}
proving \eqref{eq:prop descent 0}. Together with \Cref{lem:valid complexity} and $\kappa_\delta\geq1$, 
the above inequality also implies
\begin{equation} 
f(x^{k+1}) + \frac{n\alpha_k}{4\kappa_\delta} \cdot \cG(x^k)\leq f(x^k) + \frac{3}{2}\cdot \err_k. \nonumber
\end{equation}
Applying \Cref{lem:err bound} (a) to the above inequality, we establish \eqref{eq:prop descent 1}. Next, taking conditional expectation $\Exp_k[\cdot]$ on the above inequality, and using \Cref{lem:valid complexity} and \Cref{lem:err bound} (b), we also have
\begin{align*}
    \Exp_k[f(x^{k+1})] + \frac{n\alpha_k}{4\kappa_\delta} \cdot \cG(x^k) \leq f(x^k) + 6\kappa_\delta^3 \sL^2 n^3\alpha_k^3\cdot \cG(x^k) + \frac{6(\kappa_\delta  \sL\sG)^2}{\mu} \cdot n^2\alpha_k^3.
\end{align*}
As $\alpha_k \leq 1/(7n\sL\kappa_\delta^{2})$, we know $6\kappa_\delta^3 \sL^2 n^3\alpha_k^3 \leq n\alpha_k/(8\kappa_\delta)$. Hence, the above estimate  simplifies to
\[
 \Exp_k[f(x^{k+1})] + \frac{n\alpha_k}{8\kappa_\delta} \cdot \cG(x^k) \leq f(x^k) + \frac{6(\kappa_\delta  \sL\sG)^2}{\mu} \cdot n^2\alpha_k^3.
\]
Finally, taking total expectation finalizes the proof. 
\end{proof}

\section{Complexity bounds}\label{subsec:complexity}
In this section, we derive the sample complexity for RRMD method under various sampling schemes for generating the permutations $\{\pi^k\}$. First, we consider the general arbitrary permutation that includes the incremental mirror descent (IMD) method, and then we will show how uniform shuffling scheme can further improve the result.  

\begin{theorem}\label{thm:complexity-IMD}
Let \cref{assumption 1,assumption 1.5} hold and let $\{x^k\}$ be generated by \cref{alg:shuffle} with $0\leq \alpha_k \leq \bar \alpha/n$. Let $\tilde x^T$ be randomly sampled from $\{x^1,\ldots,x^T\}$ with $\Prob(\tilde x^T=x^k) = \frac{\alpha_k}{\sum_{i=1}^T\alpha_i}$, then   
\begin{equation}
    \label{eq:thm complexity 1}
    \Exp[\cG(\tilde x^{T})] \leq  \frac{4\kappa_\delta \cdot (f(x^1)-f^*)}{{\sum}_{k=1}^{T}n\alpha_k} + \frac{3\kappa_\delta^3  \sL^2\sG^2}{\mu} \cdot \frac{{\sum}_{k=1}^{T}n^3 \alpha_k^3}{{\sum}_{k=1}^{T}n\alpha_k}.
\end{equation} 
In particular, if we choose a constant stepsize $\alpha_k\equiv\frac{1}{n T^{1/3}}$ and require $T\geq  \max\Big\{8\sL^3\kappa_\delta^3,\frac{8\sG^3}{\delta^3}\Big\}$, then $\Exp[\cG(\tilde x^{T})] \leq \cO(T^{-2/3})$, corresponding to an $\cO(n\epsilon^{-1.5})$ sample complexity to ensure $\mathbb{E}[\cG(\tilde{x}^T)]\leq \epsilon$.
\end{theorem}

\begin{proof}
Telescoping \eqref{eq:prop descent 1} from $k=1,\ldots,T$ yields
\[
 \frac{1}{4\kappa_\delta}{\sum}_{k=1}^{T}n\alpha_k\cdot \cG(x^k) \leq (f(x^1) - f(x^{T+1})) +  \frac{3(\kappa_\delta  \sL\sG)^2}{4\mu}  {\sum}_{k=1}^{T}n^3 \alpha_k^3.
\]
Since $f(x) \geq f^*$ for all $x\in\cZ$, after multiplying  $4\kappa_\delta/(\sum_{k=1}^{T}n\alpha_k)$ on both sides, we have
\begin{equation} 
\Exp[\cG(\tilde x^{T})] =\frac{{\sum}_{k=1}^{T}n\alpha_k\cdot \cG(x^k)}{{\sum}_{k=1}^{T}n\alpha_k} \leq \frac{4\kappa_\delta \cdot (f(x^1)-f^*)}{{\sum}_{k=1}^{T}n\alpha_k} + \frac{3\kappa_\delta^3  \sL^2\sG^2}{\mu} \cdot \frac{{\sum}_{k=1}^{T}n^3 \alpha_k^3}{{\sum}_{k=1}^{T}n\alpha_k}, \nonumber
\end{equation} 
which proves \eqref{eq:thm complexity 1}. Finally, setting $\alpha_k = \frac{1}{nT^{1/3}}$ in the above inequality yields 
\[
\Exp[\cG(\tilde x^{T})] \leq  \Big[4(f(x^1)-f^*) + \frac{3(\kappa_\delta \sL\sG)^2}{\mu} \Big]\cdot \frac{\kappa_\delta}{T^{2/3}} = \cO(T^{-2/3}),
\]
under the condition that $\alpha_k=\frac{1}{nT^{1/3}}\leq \frac{\bar\alpha}{n}$, or equivalently, $T\geq  \max\Big\{8\sL^3\kappa_\delta^3,\frac{8\sG^3}{\delta^3}\Big\}$. With $n$ samples consumed in each epoch, the sample complexity for finding an $\epsilon$ solution is $nT = \cO(n\epsilon^{-1.5})$.
\end{proof}
\vspace{2mm}

It can be observed that when the permutations $\{\pi^k\}$ are arbitrarily selected, including IMD, the algorithm can suffer an $\cO(n)$ linear dependence in the worst case. In the next theorem, we show how a uniform permutation can improve it to a much better $\cO(\sqrt{n})$ sublinear dependence.

\begin{theorem}\label{thm:complexity-RRMD}
Let \cref{assumption 1,assumption 1.5} hold and let $\{x^k\}$ be generated by \cref{alg:shuffle} with $0\leq \alpha_k \leq \min\left\{\frac{1}{7n\sL\kappa_\delta^2},\frac{\delta}{2n\sG}\right\}$. Let $\tilde x^T$ be randomly selected from $\{x^1,\ldots,x^T\}$ with $\Prob(\tilde x^T=x^k) = \frac{\alpha_k}{\sum_{i=1}^T\alpha_i}$. If in addition, the permutations $\{\pi^k\}$ are independently generated via uniform shuffling, then  
\begin{equation}
    \label{eq:thm complexity 2}
\Exp[\cG(\tilde x^{T})] \leq  \frac{8\kappa_\delta \cdot (f(x^1)-f^*)}{{\sum}_{k=1}^{T}n\alpha_k} + \frac{48\kappa_\delta^3  \sL^2\sG^2}{\mu}  \cdot \frac{{\sum}_{k=1}^{T}n^2 \alpha_k^3}{{\sum}_{k=1}^{T}n\alpha_k}.
\end{equation} 
If we choose constant step sizes $\alpha_k\equiv\frac{1}{n^{2/3} T^{1/3}}$ and require $T\geq  n\cdot\max\Big\{(7\sL\kappa_\delta^2)^3,\frac{8\sG^3}{\delta^3}\Big\}$, then we have $\Exp[\cG(\tilde x^{T})] \leq \cO(n^{-\frac{1}{3}}T^{-\frac{2}{3}})$, corresponding to an $\cO(n^{0.5}\epsilon^{-1.5})$ sample complexity to ensure $\mathbb{E}[\cG(\tilde{x}^T)]\leq \epsilon$.
\end{theorem}
\begin{proof} Similar to the proof of \cref{eq:thm complexity 2}, telescoping \eqref{eq:prop descent 2} and then dividing ${\sum}_{k=1}^{T} n\alpha_k$ on both sides of the inequality yields
\[
\Exp[\cG(\tilde x^{T})] = \frac{{\sum}_{k=1}^{T}n\alpha_k\cdot \Exp[\cG(x^k)]}{{\sum}_{k=1}^{T}n\alpha_k} \leq \frac{8\kappa_\delta \cdot (f(x^1)-f^*)}{{\sum}_{k=1}^{T}n\alpha_k} + \frac{48\kappa_\delta^3  \sL^2\sG^2}{\mu}  \cdot \frac{{\sum}_{k=1}^{T}n^2 \alpha_k^3}{{\sum}_{k=1}^{T}n\alpha_k}.
\]
Then, setting $\alpha_k = \frac{1}{n^{2/3} T^{1/3}}$ in the above inequality gives 
\[
\Exp[\cG(\tilde x^{T})]  \leq  \Big[8(f(x^1)-f^*) + \frac{48(\kappa_\delta \sL\sG)^2}{\mu} \Big]\cdot \frac{\kappa_\delta}{n^{1/3}T^{2/3}} = \cO(n^{-\frac{1}{3}}T^{-\frac{2}{3}}), 
\]
as long as we require $\frac{1}{n^{2/3} T^{1/3}}\leq \min\big\{\frac{1}{7n\sL\kappa_\delta^2},\frac{\delta}{2n\sG}\big\}$, or equivalently, $T\geq  n\cdot\max\Big\{(7\sL\kappa_\delta^2)^3,\frac{8\sG^3}{\delta^3}\Big\}$. The sample complexity for finding an $\epsilon$ solution is $nT = \cO(n^{0.5}\epsilon^{-1.5})$.
\end{proof}
\vspace{2mm}

Recall that SMD requires $\cO(\varepsilon^{-2})$ samples to achieve $\Exp[\cG(\tilde x^{T})]\leq \varepsilon$, while RRMD has the bound $\cO(n^{0.5} \varepsilon^{-1.5})$. As a remark, whenever the accuracy $\varepsilon \leq \cO(n^{-1})$, RRMD can outperform SMD. We should also notice that the complexities in \Cref{thm:complexity-IMD,thm:complexity-RRMD} depend on the condition number $\kappa_\delta$ of the kernel $h$, and implicitly depend on parameter dual space radius parameter $\delta$ (which affects the range of step size $\alpha_k$). Our remark is that the condition number $\kappa_\delta$ is usually moderate for constant level $\delta = \cO(1)$. For instance, when the kernel $h$ is the Boltzmann-Shannon or Fermi-Dirac entropy, $\kappa_\delta$ is bounded above by $\exp(\delta)$, see Appendix \ref{app:kernel}. Then setting $\delta = 1$ yields $\kappa_\delta \leq \exp(1)\approx 2.7$. For power kernels, the dependence of $\kappa_\delta$ on $\delta$ is polynomial, which is even milder.

\section{Last-iterate convergence}\label{sec:last-iterate}
In this section, we establish the last-iterate convergence of RRMD, which further extend the complexity results and provides asymptotic guarantee for our method. In particular, we will work with the additional assumptions:

\begin{assumption}\label[assumption]{assumption: tame}
    Function $f:\Rd\to\R\cup\{+\infty\}$ is definable in o-minimal structure. 
\end{assumption}

Definable functions arise in many areas of optimization and geometry. For instance, any real semialgebraic or globally sub-analytic function (e.g. polynomials) is definable in the o-minimal structure \cite{kur98,loj63,gabrielov1996complements}. Functions in log-exp structures \cite{van1998tame} are also definable. Recent work \cite{davis2020stochastic} shows that loss landscapes of common neural network architectures are definable. 
For additional examples and closure properties, please see \cite{AttBolRedSou10,bolte2014proximal,davis2020stochastic,li2018calculus}.

\begin{assumption}\label[assumption]{assumption: bounded}
    Iterates $\{x^k\}$ generated by \cref{alg:shuffle} are bounded in the dual space, i.e., the sequence $\{\nabla h(x^k)\}$ is bounded.
\end{assumption}

For standard smooth nonconvex optimization, the primal boundedness assumption (the boundedness of primal iterates $\{x^k\}$) is commonly adopted in the analysis of last-iterate convergence \cite{AttBolRedSou10,bolte2014proximal,bolte2018first,li2023convergence}.   For mirror descent analysis under the relative smooth setting, Bolte el al. \cite{bolte2018first} inherit this assumption, and, for technical reason, they also require the domain of the kernel  $\dom{h} = \mathbb{R}^d$ to cover the whole space (\cite[Assumption D]{bolte2018first}). Such a requirement essentially excludes many most widely used kernels such as Boltzmann-Shannon entropy and Burg's entropy whose $\dom{h} \neq \mathbb{R}^d$. To fix this issue, we modify the boundedness assumption to dual space iterates $\{\nabla h(x^k)\}$, extending the convergence theory to a much wider class of kernels. Nevertheless, for power/polynomial kernel or exponential kernel with $\dom{h} = \mathbb{R}^d$, the image $\nabla h(\Rd) = \Rd$ and \cref{assumption: bounded} reduces to the existing primal boundedness requirement (\cite[Assumption D]{bolte2018first}). 

\begin{theorem}
    \label{thm:last-iterate convergence}
    Suppose \cref{assumption 1,assumption: tame,assumption: bounded} hold. Let $\{x^k\}$ be generated by \cref{alg:shuffle} with $\{\alpha_k\}$ satisfying ${\sum}_{k=1}^\infty \alpha_k=\infty$ and  ${\sum}_{k=1}^\infty \alpha_k^3<\infty.$ Then, the following assertions are true:
    \begin{enumerate}[label=\textup{\textrm{(\alph*)}},leftmargin=1.8em,topsep=2pt,itemsep=.5ex,partopsep=0ex]
    \item Every accumulation point of $\{x^k\}$ is stationary, i.e., $\lim_{k\to\infty}\|\nabla f(x^k)\|=0$.
    \item Moreover, if ${\sum}_{k=1}^\infty\; \alpha_k ({\sum}_{i=k}^\infty \alpha_i^3)^{\theta} <\infty$ for some $\theta\in(0,1)$, then $\sum_{k=1}^\infty \rho_h(x^k,x^{k+1})<\infty$, and $x^k\to \bar x\in \crit(f):=\{x\in\mathcal{Z}:\nabla f(x)=0\}$.
    \end{enumerate}
\end{theorem}
\begin{remark}
    The step-size conditions in \Cref{thm:last-iterate convergence} are satisfied by the step-size $\alpha_k= 1/k^\gamma$ with $\gamma\in(\frac12,1]$. It is straightforward to verify $\sum_{k=1}^\infty \alpha_k = \infty$ and $\sum_{k=1}^\infty \alpha_k^3 < \infty$. To confirm that $\{\alpha_k\}_k$ also satisfies $\sum_{k=1}^\infty \alpha_k (\sum_{i=k}^\infty \alpha_i^3)^{\theta} < \infty$, we apply the integral test:
    \begin{align*}
    {\sum}_{k=1}^\infty k^{-\gamma} \big({\sum}_{i=k}^\infty i^{-3\gamma}\big)^{\theta}  \leq  {\sum}_{k=1}^\infty k^{-\gamma} \Big(k^{-3\gamma} + \frac{1}{3\gamma-1}\cdot k^{1-3\gamma}\Big)^{\theta}  = \cO\Big({\sum}_{k=1}^\infty k^{-\gamma+(1-3\gamma)\theta} \Big) .  
    \end{align*}
    Therefore, the summation ${\sum}_{k=1}^\infty k^{-\gamma+(1-3\gamma)\theta}$ is finite whenever $\theta\in(\frac{1-\gamma}{3\gamma-1},1)$.\vspace{0.15cm}
\end{remark}

Next, let us present the last iterate convergence analysis of \cref{thm:last-iterate convergence}, which is organized as two subsections due to the complexity of the proof. 

\subsection{Proof of \Cref{thm:last-iterate convergence} (a)}
    Since $\{\nabla h(x^k)\}$ is bounded, there exist $\cX\subseteq\mathcal{Z}$ and $\sD>0$ such that $\diam_h(\cX) \leq \sD$ and $\{x^k\}\subseteq \cX$. Whence, by DKC regularity, there is a constant $\kappa_\sD\geq 1$ such that $\kappa(\cX) \leq \kappa_\sD$. We also denote $\bar \cL:=\cL(\cX)$.
   Moreover, it follows from \Cref{lem:Lipschitz continuity} that 
\begin{equation}\label{eq:bounded Lipschitz}
    \|\nabla f(x) - \nabla f(y)\| \leq \sL\sqrt{\kappa_\sD} \cdot \rho_h(x,y)\quad \text{for all}\quad x,y\in\cX.
\end{equation}
Since $\alpha_k \to 0$, there is $\bar k\in\N$ such that $\alpha_k \leq \bar \alpha/n$ for all $k\geq \bar k$. Without loss of generality, we may discard all iterates up to the $\bar{k}$-th and relabel the sequence so that $\alpha_k \leq \bar{\alpha}/n$ hold for the new sequence. Thus, the conditions of \Cref{prop:descent property} are satisfied. Summing \eqref{eq:prop descent 0} and \eqref{eq:prop descent 1} yields
\begin{equation}\label{eq:asymptotic-1}
    2f(x^{k+1}) + \frac{\cD(x^k,x^{k+1})}{2n\alpha_k} + \frac{n\alpha_k}{4\kappa_\delta}\cdot \cG(x^k) \leq  2f(x^k) + \frac{5(\kappa_\delta  \sL\sG)^2}{4\mu} \cdot n^3 \alpha_k^3.
\end{equation}
By \cref{lem:important bounds}, we have $\cD(x^k,x^{k+1}) \geq \rho_h^2(x^k,x^{k+1})/(2\bar \cL)$ and \[\cG(x^k)=\frac{\cD(x^k,\hat x^k)}{(n\alpha_k)^2}  \geq \frac{\rho_h^2(x^k,\hat x^k)}{2\bar \cL(n\alpha_k)^2} = \frac{\|\nabla f(x^k)\|^2}{2\bar \cL}, \]
where the last equality is due to \eqref{eq:update proxy}. With the above estimates, we rewrite \eqref{eq:asymptotic-1} as
\begin{equation}\label{eq:asymptotic-1.5}
    f(x^{k+1}) + \sC_1\cdot \frac{\rho_h^2(x^k,x^{k+1})}{n\alpha_k} + \sC_2 \cdot  n\alpha_k\|\nabla f(x^k)\|^2 \leq  f(x^k) + \sC_3 \cdot  n^3 \alpha_k^3,
\end{equation}
where we denote $\sC_1 :=1/(8\bar \cL)$, $\sC_2:=1/(16\bar \cL\kappa_\delta)$ and $\sC_3:=5(\kappa_\delta  \sL\sG)^2/(8\mu)$ 
for notational simplicity. Telescoping \eqref{eq:asymptotic-1.5} from $k=1$ to $\infty$ and noting that $f(x^k)>-\infty$ and $\sum_{k=1}^\infty n^3 \alpha_k^3 <\infty$, we conclude
\begin{equation}\label{eq:asymptotic-2}
   {\sum}_{k=1}^\infty \frac{\rho_h(x^k,x^{k+1})}{n\alpha_k} <\infty \quad \text{and} \quad {\sum}_{k=1}^\infty n\alpha_k\|\nabla f(x^k)\|^2 <\infty.
\end{equation}
Notice that \eqref{eq:asymptotic-2} and $\sum_{k=1}^\infty \alpha_k = \infty$ imply $\liminf_{k\to\infty} \|\nabla f(x^k)\| = 0$. To show $\lim_{k\to\infty} \|\nabla f(x^k)\| = 0$, we prove by contradiction. Specifically, we assume on the contrary that $\limsup_{k\to\infty}\|\nabla f(x^k)\| > 0$. Hence, there exist $\varepsilon>0$ and infinite subsequences $\{t_k\}_k,\{\ell_k\}_k\subseteq \N$ such that $t_k < \ell_k <t_{k+1}$ and 
\begin{equation}\label{eq:asymptotic-3}
\|\nabla f(x^{t_k})\| \geq 2\varepsilon,\quad \|\nabla f(x^{\ell_k})\|<\varepsilon\quad \text{and}\quad  \|\nabla f(x^i)\|\geq \varepsilon\quad \forall\, i\in[t_k,\ell_k)\cap\N.
\end{equation}
Denote $\beta_k:={\sum}_{i=t_k}^{\ell_k-1} n \alpha_i$. It follows from \eqref{eq:asymptotic-2} and \eqref{eq:asymptotic-3} that
\[ \infty > {\sum}_{k=1}^\infty n\alpha_k\|\nabla f(x^k)\|^2 \geq {\sum}_{k=1}^\infty {\sum}_{i=t_k}^{\ell_k-1} n\alpha_i\|\nabla f(x^i)\|^2 \geq \varepsilon^2 \,{\sum}_{k=1}^\infty \beta_k \quad 
\Longrightarrow \quad \lim_{k\to\infty} \beta_k=0.\]
By triangle inequality of $\rho_h(\cdot,\cdot)$ and Cauchy-Schwarz inequality, we have
\begin{align*}
\rho_h(x^{t_k},x^{\ell_k}) &\leq {\sum}_{i=t_k}^{\ell_k-1} \rho_h(x^{i},x^{i+1}) = {\sum}_{i=t_k}^{\ell_k-1}\sqrt{n\alpha_i} \cdot \frac{\rho_h(x^{i},x^{i+1})}{\sqrt{n\alpha_i}} \\[2mm]&\leq \Big[{\sum}_{i=t_k}^{\ell_k-1}n\alpha_i \Big]^{1/2} \cdot \bigg[ {\sum}_{i=t_k}^{\ell_k-1}\frac{\rho_h^2(x^{i},x^{i+1})}{n\alpha_i} \bigg]^{1/2} \leq \sqrt{\beta_k}\cdot \bigg[ {\sum}_{i=1}^{\infty}\frac{\rho_h^2(x^{i},x^{i+1})}{n\alpha_i}\bigg]^{1/2}.
\end{align*}
Since $\beta_k \to 0$ and ${\sum}_{i=1}^{\infty}\rho_h^2(x^{i},x^{i+1})/(n\alpha_i)<\infty$ (by \eqref{eq:asymptotic-2}), we obtain $\rho_h(x^{t_k},x^{\ell_k})\to0$. In view of this, and upon invoking \eqref{eq:bounded Lipschitz} and \eqref{eq:asymptotic-3}, a contradiction arises:
\[
\varepsilon\leq \|\nabla f(x^{t_k})\| - \|\nabla f(x^{\ell_k})\| \leq \|\nabla f(x^{t_k}) -\nabla f(x^{\ell_k})\| \leq \sL\sqrt{\kappa_\sD} \cdot \rho_h(x^{t_k},x^{\ell_k}) \to 0.
\]
Thus, $\lim_{k\to\infty} \|\nabla f(x^k)\| = 0$, i.e., every accumulation point of $\{x^k\}$ is stationary. \hfill $\Box$

\subsection{Proof of \Cref{thm:last-iterate convergence} (b)}
To proceed, we denote $\cA$ the set of accumulation points. By \Cref{thm:last-iterate convergence} (a), $\cA \subseteq \crit(f)$. Next, we introduce the Kurdyka-{\L}ojasiewicz (KL) property for definable functions. This is simplified version of \cite[Lemma 4.11]{josz2024proximal}, see also, \cite[Lemma 4.1]{qiu2025normal}.
\begin{lemma}[The KL inequality]
\label[lemma]{lem:KL}
    Let \Cref{assumption: tame} hold and let $\bar{x}\in \crit(f)$ be given. For all $\theta\in(0,1)$, there then are $\varepsilon,\varsigma > 0, \eta\in(0,1]$ and a continuous, concave function $\psi:[0,\eta) \to \R_+$, which is $C^1$ on $(0,\eta)$, satisfying 
    \begin{equation}
         \label{eq:desingularization function}
         \psi(0)=0\quad \text{and} \quad 
    1/\psi^\prime(s+t) \leq 1/\psi^\prime(s) + \varsigma t^{\theta} \quad \text{for all $s,t >0$ with $s+t<\eta$}
    \end{equation}
     such that, for all $x\in\cB_\varepsilon(\bar x) \cap\{x:0<|f(x)-f(\bar x)|<\eta\}$, the KL inequality holds:
     \begin{equation}
         \label{eq:mer-kl}
         \psi^\prime(|f(x)-f(\bar x)|) \cdot  \|\nabla f(x)\|\geq 1.
     \end{equation} 
\end{lemma}
\begin{proof}{Proof of \Cref{thm:last-iterate convergence} (b).}
First, we introduce an auxiliary sequence \[\Upsilon_k:=f(x^k) + \sC_3 {\sum}_{i=k}^\infty n^3 \alpha_i^3,\] which is well defined as $\sum_{i=1}^\infty n^3 \alpha_i^3<\infty$. Then, \eqref{eq:asymptotic-1.5} is rewritten as
\begin{equation}\label{eq:asymptotic-4}
    \Upsilon_{k+1} + \sC_1\cdot \frac{\rho_h^2(x^k,x^{k+1})}{n\alpha_k} + \sC_2 \cdot  n\alpha_k\|\nabla f(x^k)\|^2 \leq  \Upsilon_k.
\end{equation}
Clearly, the sequence $\{\Upsilon_k\}_k$ is non-increasing. Since it is bounded below, there exists a constant $\bar{f}\in\R$ such that $\Upsilon_k \to \bar{f}$ as $k \to \infty$. Moreover, $\lim_{k\to\infty}\sum_{i=k}^\infty n^3 \alpha_i^3=0$ implies that $\lim_{k\to\infty} f(x^k)= \lim_{k\to\infty} \Upsilon_k = \bar{f}$. Note also that $f$ is continuous, we have $\bar{f}=f(x)$ for all $x\in\cA$. 

By the uniformized KL property \cite[Lemma 6]{bolte2014proximal}, \Cref{lem:KL} and $\cA\subseteq\crit(f)$, there are $\varepsilon,\eta>0$ such that the KL inequality \eqref{eq:mer-kl}---with the desingularization function $\psi:[0,\eta)\to\R_+$ specified in \Cref{lem:KL}---holds for all $x\in\mathcal{Z}$ satisfying
\begin{equation}
    \label{eq:KL region}
    \dist(x,\cA)<\varepsilon \quad \text{and} \quad |f(x)-\bar f|<\eta.  
\end{equation}
Since $\cA$ is accumulation points set, $\{\Upsilon_k\}_k$ is non-increasing, and $\lim_{k\to\infty} f(x^k)= \lim_{k\to\infty} \Upsilon_k = \bar{f}$, there exists $\tilde k\in\N$ such that $x^k$ satisfies \eqref{eq:KL region} and $\Upsilon_k-\bar f\in[0,\eta)$ for all $k\geq \tilde k$. Without loss of generality, we may assume $\Upsilon_k\neq\bar f$ (otherwise, $x^k$ is a stationary point by \eqref{eq:asymptotic-4}) and work with $k\geq \tilde k$ in the subsequent analysis. 

By concavity of $\psi$, we have 
\begin{equation}
    \label{eq:asymptotic-5}
    \psi(\Upsilon_{k}-\bar f)-\psi(\Upsilon_{k+1}-\bar f)\geq \psi^\prime(\Upsilon_{k}-\bar f) \cdot (\Upsilon_{k} - \Upsilon_{k+1}).
\end{equation}
It follows from \eqref{eq:asymptotic-4} and $2a^2+2b^2 \geq (a+b)^2$ that
\begin{equation}
    \label{eq:asymptotic-6}
    \begin{aligned}
           \Upsilon_{k} - \Upsilon_{k+1} &\geq \sC_1\cdot \rho_h^2(x^k,x^{k+1})/(n\alpha_k) + \sC_2 \cdot  n\alpha_k\|\nabla f(x^k)\|^2\\[2mm]
           &=\frac{\sC_2}{2n\alpha_k} \cdot \Big[2(n\alpha_k\|\nabla f(x^k)\|)^2 + 2\big(\sqrt{\sC_1/\sC_2}\cdot \rho_h(x^k,x^{k+1})\big)^2 \Big] \\[2mm]
           &\geq \frac{\sC_2}{2n\alpha_k} \cdot \Big[n\alpha_k\|\nabla f(x^k)\| + \sqrt{\sC_1/\sC_2}\cdot \rho_h(x^k,x^{k+1}) \Big]^2.
    \end{aligned}
\end{equation}
On the other hand, based on  properties of $\psi$ in \Cref{lem:KL}, we deduce
\begin{equation}
    \label{eq:asymptotic-7}
    \begin{aligned}
         \psi^\prime(\Upsilon_{k}-\bar f) & \overset{\text{(i)}}{\geq} \psi^\prime\big(|f(x^k)-\bar f| + \sC_3 {\textstyle \sum}_{i=k}^\infty n^3 \alpha_i^3\big)\overset{\text{(ii)}}{\geq} \Big[1/\psi^\prime(|f(x^k)-\bar{f}|) + \varsigma\sC_3^{\theta} \big({\sum}_{i=k}^\infty n^3 \alpha_i^3\big)^{\theta}  \Big]^{-1} \\[2mm] &\overset{\text{(iii)}}{\geq} \Big[\|\nabla f(x^k)\| + \varsigma\sC_3^{\theta} \big({\sum}_{i=k}^\infty n^3 \alpha_i^3\big)^{\theta}  \Big]^{-1},  
    \end{aligned}
\end{equation}
where (i) utilizes the monotonicity of $\psi^\prime$ (since $\psi$ is concave), (ii) follows from \eqref{eq:desingularization function}, and (iii) holds due to the KL inequality \eqref{eq:mer-kl}. Inserting estimates \eqref{eq:asymptotic-6}, \eqref{eq:asymptotic-7} into \eqref{eq:asymptotic-5}, we then obtain
\begin{align*}
 \psi(\Upsilon_{k}-\bar f)-&\psi(\Upsilon_{k+1}-\bar f) \geq \frac{\sC_2}{2} \cdot  \frac{\big[n\alpha_k\|\nabla f(x^k)\| + \sqrt{\sC_1/\sC_2}\cdot \rho_h(x^k,x^{k+1}) \big]^2}{n\alpha_k\|\nabla f(x^k)\| + \sC_3^{\theta}\varsigma \cdot  n\alpha_k ({\sum}_{i=k}^\infty n^3 \alpha_i^3)^{\theta}}\\[2mm]
&\geq \frac{\sC_2}{2} \cdot  \Big[ n\alpha_k\|\nabla f(x^k)\| + \sqrt{\sC_1/\sC_2}\cdot \rho_h(x^k,x^{k+1}) - \sC_3^{\theta}\varsigma \cdot  n\alpha_k \big({\sum}_{i=k}^\infty n^3 \alpha_i^3\big)^{\theta}\Big],
\end{align*}
where the last line uses $\frac{(a+b)^2}{a+c} \geq \frac{(a+b)^2}{a+b+c} \geq a+b-c$ for all $a,b,c\geq0$ with $a+c>0$. Rearranging this recursion and telescoping, along with $\lim_{k\to\infty}\Upsilon_k=\bar f$, we have 
\[
\sqrt{\sC_1\sC_2}\; \sum_{k=\tilde k}^\infty\;  \rho_h(x^k,x^{k+1}) + \sC_2\,  \sum_{k=\tilde k}^\infty\; n\alpha_k \|\nabla f(x^k)\|\leq 2\psi(\Upsilon_{\tilde k}) + \sC_2\sC_3^{\theta}\varsigma \cdot  \sum_{k=\tilde k}^\infty\; n\alpha_k \Big({\sum}_{i=k}^\infty n^3 \alpha_i^3 \Big)^{\theta} <\infty, 
\]
where the last inequality is due to the condition that 
${\sum}_{k=1}^\infty\; n\alpha_k ({\sum}_{i=k}^\infty n^3 \alpha_i^3)^{\theta} <\infty$. That is, 
\[
{\sum}_{k=\tilde k}^\infty\;  \rho_h(x^k,x^{k+1})<\infty,
\]
and $\{\nabla h(x^k)\}$ is a Cauchy sequence in the dual sequence. Therefore the limit $\lim_{k\to\infty}\nabla h(x^k)=\bar{y}$ exists. By the essential smoothness of $h$, there exist $\bar{x}\in\idom{h}\subseteq\mathcal{Z}$ s.t. $\bar{x} = \nabla h^{*}(\bar{y})$ and $\lim_{k\to\infty}x^k = \bar{x}$. Along with \Cref{thm:last-iterate convergence} (a), we show $x^k\to\bar x$ for some $\bar x\in\crit(f)$. 
\end{proof}

\section{Further discussions}
In the previous analysis of RRMD, we adopt several simplification on problem setting in order to prevent key novelty, the newly proposed dual Lipschitz continuity and DKC regularity, to be buried under technicalities. In this section, we will illustrate how these requirements can be relaxed.

\subsection{Relaxing the bounded gradients assumption}\label{subsec:relax bdd grad}
Throughout our analysis, we employ a bounded gradient assumption (see \cref{assumption 1.5}). While this can be relaxed to a more general expected smoothness condition \eqref{eq:expected smoothness}, adopting this generalization would significantly complicate the derivations and potentially obscure the central focus. Therefore, for clarity of exposition, we maintain the stronger bounded gradient assumption. In this subsection, we provide a proof sketch demonstrating how the general condition can be incorporated into the RRMD analysis.

The condition is formulated as follows: There exist constants $\sA,\sB,\tau\geq0$ such that 
\begin{equation}\label{eq:expected smoothness}
    \frac1n{\sum}_{i=1}^n\|\nabla f_i(x)\|^2 \leq \sA \left(f(x) - \underline{f}\right)^\tau + \sB^2 \quad  \forall x\in\mathcal{Z}
\end{equation}
where $\underline{f}$ can be any lower bound of $f_i$ over all $i\in\intset{n}$. This is a generalization of the standard \emph{expected smoothness condition}   \cite{gower2019sgd,gower2021stochastic,sebbouh2021almost}, which corresponds to the special case $\tau=1$. By introducing a free parameter $\tau$, \eqref{eq:expected smoothness} captures wider behaviors, a few examples are discussed below. 
\begin{example}\label{Example:FLip} \!\!\!\!
   \emph{If \Cref{assumption 1.5} is true, then condition \eqref{eq:expected smoothness} holds with $\tau=0$. }
\end{example}
\begin{example}\label{Example:FSmooth} \!\!\!\!
   \emph{If each $f_i$ is $\sL$-smooth and lower bounded, then condition \eqref{eq:expected smoothness} holds with $\tau=1$.}
\end{example}  
\begin{example}\label{Example:poly}
\!\!\!\!
\emph{For quadratic inverse problem where each $f_i(x):=\frac14\big(x^\top M_i x - b_i\big)^2$, with $M_i\succeq0$ and $ b_i\in\R_+$ being known data points, then condition \eqref{eq:expected smoothness} holds with $\tau = 3/2$.} 
\end{example}
\vspace{0.2cm}
\textbf{Proof sketch under the relaxed assumption \eqref{eq:expected smoothness}.} We begin with redefining the sets $\cX^k$ and the step-size upper bound $\bar \alpha_k$, previously defined in \eqref{def:Xk}, as follows: 
\begin{align*}
\cX^k&:=\bigg\{u\in\idom{h}: \rho_h(u,x^k) \leq \frac{\delta}{2n\nu} {\sum}_{i=1}^n\|\nabla f_{i}(x^k)\| \bigg\} 
    \quad \text{and}  \quad \bar \alpha_k:=\min\bigg\{ \frac{1}{2\sL \kappa(\cX^k)},\frac{\delta}{2\nu}\bigg\}
\end{align*}
where $\nu := \sqrt{\sA[\sC(f(x^1)-\underline{f})]^\tau+\sB^2}$ for some $\sC\in\R_+$. Above, $\nu$ plays the role of the gradients upper bound $\sG$. The main challenge is to derive a uniform bound on $\kappa(\cX^k)$ for all $k\in\N$ without knowing $\|\nabla f_i(x^k)\|$ a priori. We proceed via carefully structured steps:

\noindent\emph{Step 1: Update region.} We prove by induction that when $\alpha_k \leq \bar{\alpha}_k/n$, the intermediate iterates $y^{k,i} \in \cX^k$, $\forall i \in \intset{n+1}$. The key is to show $\rho_h(y^{k,i},x^k) \leq 2\alpha_k \sum_{j=1}^{i-1} \|\nabla f_{\pi^k_j}(x^k)\| \leq 2\alpha_k \sum_{j=1}^{n} \|\nabla f_{j}(x^k)\|$.

\noindent\emph{Step 2: Diameter control.} The diameter of $\cX^k$ can be bounded as \[\diam_h(\cX^k) \leq \frac{\delta}{n\nu}{\sum}_{i=1}^n\|\nabla f_{i}(x^k)\| \leq \frac{\delta}{\nu}\sqrt{\sA [f(x^k)-\underline{f}]^\tau + \sB^2}.\]

\noindent\emph{Step 3: Function bound.} Under $\alpha_k \leq \bar \alpha_k/n$, we derive the following bound via careful error estimates of $\err_k$ and descent-type property of $f$, 
\[
f(x^{k+1}) - \underline{f} \leq \exp\Big\{\beta\kappa(\cX^k)^3 \alpha_k^3 \Big\} \cdot [f(x^k)-\underline{f}]\quad \text{for some $\beta>0$.}
\]
\noindent\emph{Step 4: Simultaneous induction.} We establish the following claim for all $k \in \intset{T}$ by induction:
\begin{align*}
f(x^k) - \underline{f} \leq \exp\Big\{\beta\kappa_\delta^3 {\sum}_{i=1}^{k-1} \alpha_i^3 \Big\} \cdot [f(x^1)-\underline{f}] \leq \sC[f(x^1)-\underline{f}], \quad 
\kappa(\cX^k) \leq \kappa_\delta \quad \text{and} \quad  \bar{\alpha} \leq \bar{\alpha}_k
\end{align*}
for some constant $\sC>0$, where $\bar \alpha:=\min\{\frac{1}{2\sL\kappa_\delta},\frac{\delta}{2\nu}\}$  and $T$ denotes the number of total iterations.

\emph{Inductive step:} Assuming the above claim holds up to $k = t$. Then
\begin{itemize}
    \item It follows from Step 3, $\kappa(\cX^i)\leq \kappa_\delta$, $1\leq i\leq t$ and $\beta\kappa_\delta^3 {\sum}_{i=1}^T \alpha_i^3 \leq \log(\sC)$ that
    \[f(x^{t+1}) - \underline{f} \leq \exp\Big\{\beta\kappa_\delta^3 \alpha_t^3 \Big\} \cdot [f(x^t)-\underline{f}] \leq \exp\Big\{\beta\kappa_\delta^3 {\sum}_{i=1}^t \alpha_i^3 \Big\} \cdot [f(x^1)-\underline{f}] \leq \sC \cdot [f(x^1)-\underline{f}]. \]
    \item From Step 2 and $f(x^{t+1}) - \underline{f} \leq \sC [f(x^1) - \underline{f}]$, we have $\diam_h(\cX^{t+1}) \leq \delta$, which combined with DKC, shows $\kappa(\cX^{t+1})\leq \kappa_\delta$. By the definition of $\bar \alpha_k$, it also holds that $\bar \alpha \leq \bar\alpha_{t+1}$. 
\end{itemize}

\noindent\emph{Step 5: Recovering the previously established results.} With the uniform bound $\kappa(\mathcal{X}^k) \leq \kappa_\delta$ established and with the step-size constraint simplified to $\alpha_k \leq \bar{\alpha}$, one can safely follow the derivations in previous sections to get the complexity bound and last-iterate convergence.

\subsection{Extension to multi-block setting}
In previous sections, we focus on the single-block $h$. This section will show that the obtained results can be readily extended to the separable kernels $h(x)=\sum_{i=1}^m h_i(x_{[i]})$. Correspondingly, the vector is split into $m$ blocks: $x=(x_{[1]},\ldots,x_{[m]})$ with sizes $d_1,\ldots,d_m$ and $\sum_{i=1}^m d_i =d$, and any subset $\cX\subseteq\Rd$ is decomposed into $\cX=\cX_1\times \cdots \times \cX_m$, meaning that if $x\in\cX$, then $x_{[i]}\in\cX_i$. Moreover, we denote $\nabla_{[j]} \; f_i(\cdot)$  the partial derivatives of $f_i$ w.r.t. the $j$-th variable block. It holds that $[\nabla f_i(\cdot)]_{[j]} =\nabla_{[j]}\, f_i(\cdot)$.

In the convergence analysis of RRMD, the key ingredient is to deal with the inner product $$
\big\langle  \nabla f_{\pi^k_i}(x^k) - \nabla f_{\pi^k_i}(y^{k,i}) , x^{k+1} - x^k  \big\rangle,
$$
which, in the multi-block case, can be expressed as:
\begin{align*}
&\hspace{0.4cm}\Big\langle  \nabla f_{\pi^k_i}(x^k) - \nabla f_{\pi^k_i}(y^{k,i}), x^{k+1} - x^k  \Big\rangle\\
& = {\sum}_{j=1}^m \Big\langle \nabla_{[j]} \; f_{\pi^k_i}(x^k) - \nabla_{[j]} \; f_{\pi^k_i}(y^{k,i}),\; x^{k+1}_{[j]}  - x^k_{[j]}  \Big\rangle\\
&\leq {\sum}_{j=1}^m  \; \frac{\kappa_\delta \cdot n\alpha_k}{2\mu_j(\cX^k_j)} \big\|\nabla_{[j]} \; f_{\pi^k_i}(x^k) - \nabla_{[j]} \; f_{\pi^k_i}(y^{k,i}) \big\|^2 + {\sum}_{j=1}^m \; \frac{\mu_j(\cX^k_j)}{\kappa_\delta\cdot 2n\alpha_k}\|x^{k+1}_{[j]}  - x^k_{[j]}\|^2
    \\&\leq {\sum}_{j=1}^m \; \frac{\kappa_\delta \cdot n\alpha_k}{2\mu_j(\cX^k_j)} \big\| \nabla_{[j]} \; f_{\pi^k_i}(x^k) - \nabla_{[j]} \; f_{\pi^k_i}(y^{k,i}) \big\|^2 + \frac{\cD(x^k,x^{k+1}) + \cD(x^{k+1},x^k)}{2n\alpha_k}
\end{align*}
where the last line is because  $x^k,x^{k+1}\in\cX^k$ with $\kappa(\cX^k)\leq \kappa_\delta$ (see \Cref{lem:region}),  the bound \begin{align*}
    \|x^k_{[j]}-x^{k+1}_{[j]}\|^2 &\leq \frac{\rho_{h_j}(x^k_{[j]},x^{k+1}_{[j]})^2}{\mu_j(\cX^k_j)} \leq \frac{\cL_j(\cX^k_j)}{\mu_j(\cX^k_j)} \cdot [{\cal D}_{h_j}(x^k_{[j]},x^{k+1}_{[j]})+{\cal D}_{h_j}(x^{k+1}_{[j]},x^k_{[j]})] \\[2mm] & \leq \kappa_\delta [{\cal D}_{h_j}(x^k_{[j]},x^{k+1}_{[j]})+{\cal D}_{h_j}(x^{k+1}_{[j]},x^k_{[j]})],
\end{align*} and $\cD(x,y)=\sum_{j=1}^m {\cal D}_{h_j}(x_{[j]},y_{[j]})$ for all $x,y\in\cZ$. In view of this, we update accordingly the definitions of the stochastic errors $\err_k$ and the step-size upper bound $\bar \alpha_k$ as follows:
$$
\cE_k:=\frac{\kappa_\delta \cdot \alpha_k}{2} \sum_{i=1}^n \sum_{j=1}^m \frac{1}{\mu_j(\cX^k_j)} \| \nabla_{[j]}\, f_{\pi^k_i}(x^{k}) - \nabla_{[j]}\, f_{\pi^k_i}(y^{k,i}) \|^2\quad \text{and} \quad \bar \alpha_k:=\min\Big\{\frac{1}{2\Gamma(\cX^k)\sL\kappa_\delta},\frac{\delta}{2\sG}\Big\},
$$
where \[\Gamma(\cX)=\min  \Big\{8m, 8e\Big[2+3\ln\Big(\frac{\max_j\mu_j(\cX_j)}{\min_j\mu_j(\cX_j)}\Big)\Big]\Big\}^{1/2}.\]
Below, we establish the error estimate for multi-block kernels with derivations in Appendix \ref{proof:lem:err bound multi}.  
\begin{proposition}[Stochastic error bound in multi-block setting]
\label[proposition]{lem:err bound multi}
Under \cref{assumption 1,assumption 1.5}.
Let $\{x^k\}$ and $\{y^{k,i}\}$ be generated by \cref{alg:shuffle} with step-size $0<\alpha_k \leq \bar \alpha_k/n.$ 
 \begin{enumerate}[label=\textup{\textrm{(\alph*)}},leftmargin=3em,topsep=2pt,itemsep=.5ex,partopsep=0ex]
 \item Then, it holds that
\[\err_k \leq \frac{\Gamma^2(\cX^k)\kappa_\delta(\kappa_\delta  \sL\sG)^2}{2\mu} \cdot n^3 \alpha_k^3.\]
\item  Moreover, if $\pi^k$ is generated by uniform shuffling scheme, we have 
\[
\Exp_k[\err_k]\leq 4\Gamma^2(\cX^k)\sL^2 \kappa_\delta (\kappa_\delta  n\alpha_k)^3\cdot \cG(x^k) + \frac{4\Gamma^2(\cX^k)\kappa_\delta(\kappa_\delta  \sL\sG)^2}{\mu} \cdot n^2\alpha_k^3.
\]
 \end{enumerate}
\end{proposition}
Compared with the non-separable case (see \Cref{lem:err bound}), the upper bounds for $\err_k$ are amplified by $\Gamma^2(\cX^k)\kappa_\delta$. Hence, the approximate descent property \Cref{prop:descent property} continues to hold with the factors of $n^3\alpha_k^3$ in \eqref{eq:prop descent 0}, \eqref{eq:prop descent 1} and \eqref{eq:prop descent 2} being enlarged by $\Gamma^2(\cX^k)\kappa_\delta$. Since $\Gamma^2(\cX^k) \leq 8m$, the asymptotic results in \Cref{thm:last-iterate convergence} continue to hold. The complexity of RRMD will change due to the new estimates of stochastic errors for the multi-block kernels. 
Below, we present the new complexity bounds. The derivation mirrors that of \cref{thm:complexity-IMD,thm:complexity-RRMD}, and is omitted here. 
\begin{theorem}
Let \cref{assumption 1,assumption 1.5} hold and let $\{x^k\}$ be generated by \cref{alg:shuffle} with $0\leq \alpha_k \leq \bar \alpha_k/n$. Let $\tilde x^T$ be randomly sampled from $\{x^1,\ldots,x^T\}$ with $\Prob(\tilde x^T=x^k) = \alpha_k/(\sum_{i=1}^T\alpha_i)$. 
 \begin{enumerate}[label=\textup{\textrm{(\alph*)}},leftmargin=3em,topsep=2pt,itemsep=.5ex,partopsep=0ex]
\item It holds that
\[
    \Exp[\cG(\tilde x^{T})] \leq  \frac{4\kappa_\delta \cdot (f(x^1)-f^*)}{{\sum}_{k=1}^{T}n\alpha_k} + \max_{k\in\intset{T}}\Gamma^2(\cX^k)\cdot \frac{3\kappa_\delta^4  \sL^2\sG^2}{\mu} \cdot \frac{{\sum}_{k=1}^{T}n^3 \alpha_k^3}{{\sum}_{k=1}^{T}n\alpha_k}.
\]
 \item  If $\{\pi^k\}_k$ are generated via uniform shuffling and $\alpha_k \leq 1/[7n\sL\Gamma(\cX^k)\kappa_\delta^{5/2}]$, then
\[
\Exp[\cG(\tilde x^{T})] \leq  \frac{8\kappa_\delta \cdot (f(x^1)-f^*)}{{\sum}_{k=1}^{T}n\alpha_k} + \max_{k\in\intset{T}}\Gamma^2(\cX^k)\cdot  \frac{48\kappa_\delta^4  \sL^2\sG^2}{\mu}  \cdot \frac{{\sum}_{k=1}^{T}n^2 \alpha_k^3}{{\sum}_{k=1}^{T}n\alpha_k}.
\]
\end{enumerate}
\end{theorem}
\begin{remark}
    According to the definition of $\Gamma(\cX)$, we have
    \[\Gamma(\cX^k) \leq \sqrt{8m}\quad \text{and}\quad \Gamma(\cX^k) \leq \sqrt{8e\Big[2+3\ln\Big(\frac{\max_j\mu_j(\cX_j^k)}{\min_j\mu_j(\cX_j^k)}\Big)\Big]}.\]
    The second bound grows logarithmically with the ratio of the strong convexity parameters, it remains small even for large
ratios. For instance, even in the extremely imbalanced case where 
$\frac{\max_j\mu_{j}(\cX_j^k)}{\min_j\mu_{j}(\cX_j^k)}=10^{15}$, one still has $\Gamma(\cX^k) \leq \min\{48,2\sqrt{m}\}$. 
\end{remark}

\section{Numerical experiments}\label{sec:experiment}
In this section, we examine the performance of random reshuffling mirror descent (RRMD) and the incremental mirror descent (IMD) methods on phase retrieval and Poisson inverse problems, with stochastic mirror descent (SMD) being the benchmark. Additionally, we also test the corresponding momentum variants (denoted RRMD-M, IMD-M, SMD-M) of these methods, where the update follows
$$\begin{cases}
x^{k+1}=\argmin_z \, f(x^k)+ \langle v^{k+1},z-x^k\rangle + \frac{1}{\alpha_k}\, \cD(z,x^k),\\[1mm]
v^{k+1} \hspace{0.02cm}= \beta v^k + \alpha_k\nabla f_{i_k}(x^k) \hspace{1.343cm}\text{with}\hspace{1.34cm} v^1 := 0.
\end{cases}$$  
Here, the index $i_k\in\intset{n}$ is determined by the sampling scheme, $v^k\in\Rd$ represents the momentum term, and the momentum parameter is fixed at $\beta=0.9$ throughout this section. To measure the convergence behavior of the tested algorithms, we employ the relative error metric $\frac{f(x^k) - f(\hat x)}{f(\hat x)}$, where $\hat x$ fulfills $\|\nabla f(\hat x)\|\leq 10^{-8}$.

\subsection{Phase retrieval}
\begin{figure}[t]
    \centering
    \begin{tikzpicture}[scale=.95]
    \node[right] at (-5.5,-2) {\includegraphics[scale=.8,trim=0 0 0 0,clip]{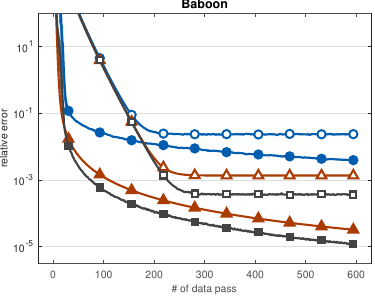}};
    \node[right] at (0.15,-2) {\includegraphics[scale=.8,trim=7 0 0 0,clip]{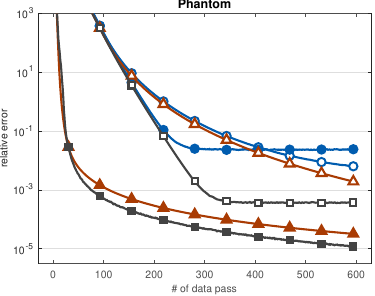}};
    \node[right] at (5.5,-2) {\includegraphics[scale=.8,trim=7 0 0 0,clip]{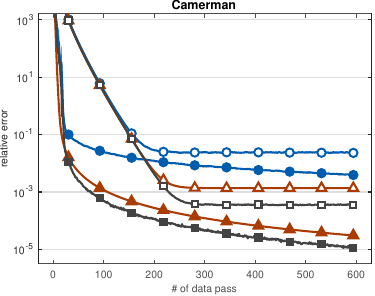}};
    \node[right] at (-2,.8) {
    \includegraphics[width=.6\textwidth]{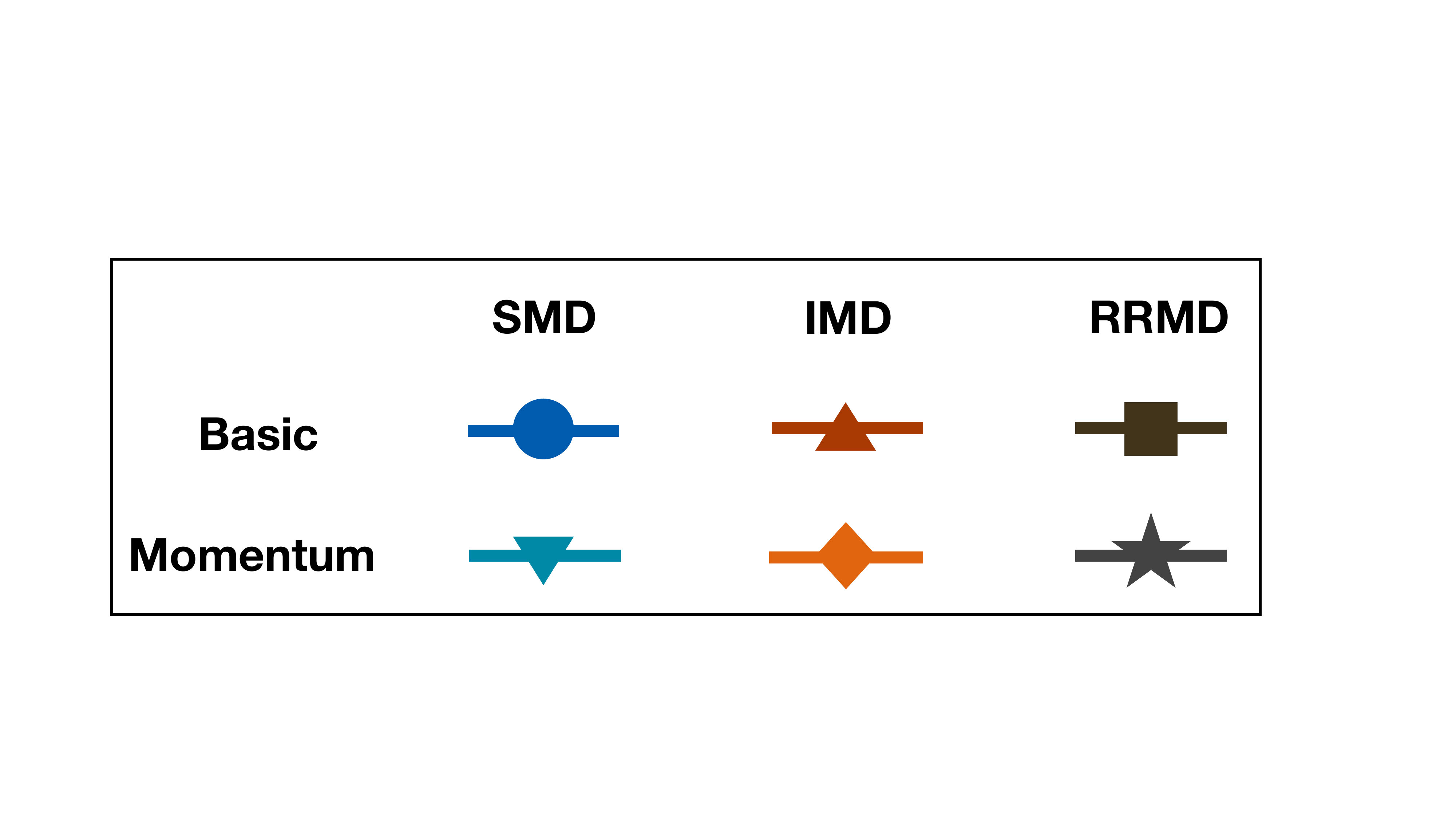}
    };
    \end{tikzpicture}
    \vspace{-.8cm}
    \caption{Numerical results on phase retrieval problem. Performance over five independent run. }
    \label{fig:exp-1}
\end{figure}
We conduct experiments on phase retrieval problems, which is a specific class of quadratic inverse problems investigated in \cite{bolte2018first}. The problem takes the form:
\[
\min_{x\in\Rd} \quad f(x):=\frac1n {\sum}_{i=1}^n (|\langle a_i,x\rangle|^2 - b_i^2)^2,
\]
where $a_{i}\in\Rd$ denote the sampling vectors and $b_{i}$ are measurement data.
According to (\cite[Lemma 5.1]{bolte2018first}), $f:\Rd\to\R$ is smooth relative to the quartic kernel $h(x):=\frac14\|x\|^4+\frac12\|x\|^2$.  

We test the above algorithms for three image signals: Baboon, Shepp-Logan phantom and Cameraman. We crop those images to $64\times64$ pixels and vectorize them into $\Rd$ vector $x_{\text{true}}$ with dimension $d=4096$. The total number of samples is $n=6d=24576$. Measurement vectors $a_i$ are drawn i.i.d. from standard normal distribution ${\cal N}(0,1)$. The observations $b_i$ are generated via $b_i=|\langle a_i,x_{\text{true}} \rangle| + e_i$ where $e_i\sim {\cal N}(0,0.1)$. For all algorithms, we maintain consistent parameters. We fix batch size to $128$ and utilize diminishing step-size $\alpha_k = \min\{10^{-5}, \alpha/k\}$, where the upper bound $10^{-5}$ is imposed to avoid divergence\footnote{In the Shepp-Logan phantom, $10^{-5}$ is still too large for SMD-M; thus, we reduce the cap to $10^{-6}$ for SMD-M.} , $k$ represents the current epoch and $\alpha$ is the tunable hyper-parameters. For fair comparison, we evaluate those algorithms across a logarithmic grid of hyper-parameter $\alpha \in\{10^{-6},10^{-5},\ldots, 10^{-1},10^0\}$ and report their best convergence results in \Cref{fig:exp-1}, where the performance is measured by the decay of the relative error as a function of the number of data passes.

In \Cref{fig:exp-1}, it can be observed that both IMD and RRMD outperforms the SMD method, while RRMD dominates IMD in all test instances. Though not analyzed in the paper, we also test the momentum variant of these paper as momentum is almost a default enhancement in the practical implementation of  stochastic first-order methods. In the experiment, incorporating the momentum provides a stable performance boost, and RRMD-M consistently
demonstrates the most favorable convergence performance.

\subsection{Poisson inverse problem}
\begin{figure}
 \begin{tikzpicture}[scale=.95]
    \node[right] at (-5.5,-2) {\includegraphics[scale=.8,trim=0 0 0 0,clip]{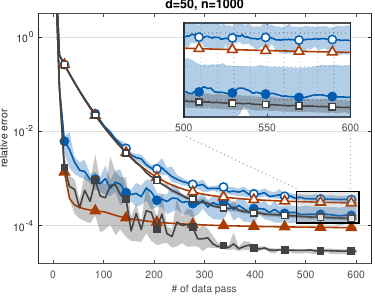}};
    \node[right] at (0.15,-2) {\includegraphics[scale=.8,trim=7 0 0 0,clip]{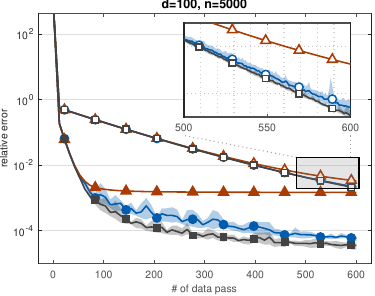}};
    \node[right] at (5.5,-2) {\includegraphics[scale=.8,trim=7 0 0 0,clip]{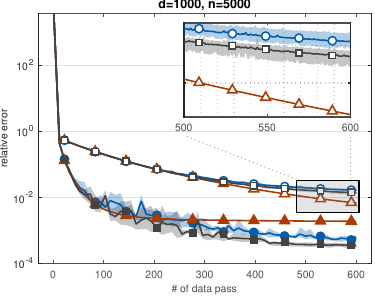}};
    \node[right] at (-2,.8) {
    \includegraphics[width=.6\textwidth]{figures/legend.pdf}
    };
    \end{tikzpicture}
    \vspace{-.8cm}
    \caption{Numerical results on Poisson inverse problem. Performance over five independent run. }
    \label{fig:exp-2}
\end{figure}

We consider the regularized Poisson inverse problem 
\[
\min_{x\in\Rd_{+}} \; f(x):= \frac1n \cdot D_{\text{KL}}(b,Ax):=\frac1n{\sum}_{i=1}^n \; (a_i^\top x) \cdot \log(a_i^\top x/b_i) - (a_i^\top x - b_i),
\]
where $D_{\text{KL}}$ denotes the generalized Kullback-Leibler (KL) divergence \cite{csiszar1991least}, $A:=[a_1^\top,\ldots, a_n^\top]^\top \in\R^{n\times d}_{++}$ and $b:=[b_1,\ldots,b_n]^\top\in\R^n_{++}$. For this instance, it can be shown that the objective function $f$ is smooth relative to the regularized Burg's entropy kernel \cite{bauschke2017descent}.

The ground-truth vector $x_{\text{true}}\in\Rd$ is sampled uniformly from the hypercube $[0,10]^d$. The measurement matrix $A=[a_1^\top,\ldots,a_n^\top]^\top\in\R^{n\times d}_+$ is sampled i.i.d. from the Student's t-distribution with 5 degrees of freedom and taking their absolute values. From each row $a_i^\top$ of $A$, we generate the corresponding observation $b_i$ from a Poisson distribution with mean parameter $a_i^\top x_{\text{true}}$. 
We conduct numerical experiments on three problem scales: $(d,n) = (50,1000)$, $(100,5000)$, and $(1000,5000)$. The definition of the relative error and the strategy for selecting the batch size and step-size follows that of the phase retrieval experiment. Specifically, we fix the batch size at $128$ and employ a diminishing step-size 
\(
\alpha_k = \min\{\gamma, \alpha/k\}
\).  
For the instance $(d,n)=(100,5000)$, we set $\gamma=0.1$ and evaluate all algorithms over a logarithmic grid of parameters $\alpha \in \{10^{-2},10^{-1},\ldots,10^{2},10^{3}\}$. 
For the other two instances $(d,n)= (50,1000)$ and $(1000,5000)$, we set $\gamma=1$ and select the optimal $\alpha$ from the same logarithmic grid. We report the performance in \Cref{fig:exp-2}.

Similar to the previous experiments, it can be observed that momentum provides a consistent performance boost for all the methods. RRMD-M dominates all the algorithms, while RRMD outperforms all the non-momentum algorithms. However, it is worth noting that IMD and IMD-M are occasionally outperformed by SMD and SMD-M, respectively. A possible reason is that IMD and IMD-M only shuffle the data once and then they compute the ``stochastic'' gradients in a deterministic and cyclic style, leading to slow improvement IMD and IMD-M happen to encounter a bad data ordering. This in turn motivates the necessity of adopting random reshuffling in each epoch to avoid the bad ordering. 
\appendix
\section{Preparatory tools}\label{app:prep tool}


We start with  the well-known three points identity \cite{chen1993convergence}. 
\begin{lemma}[Three points identity] 
\label[lemma]{lem:three point}
Let $h: \mathbb{R}^d \to (-\infty,+\infty]$ be a kernel function associated with $\mathcal{Z}$. For any
$x \in \dom{h}$, and $y,z \in \idom{h}$, we have
\[
\cD(x,z) - \cD(x,y) - \cD(y,z) = \langle\nabla h(y) - \nabla h(z), x-y\rangle.
\]
\end{lemma}

We restate the variance of sampling without-replacement from \cite[Lemma 1]{mishchenko2020random}. 
\begin{lemma}[Sampling without replacement]
\label[lemma]{lem:sample-without-re}
 Given $X^1,\dots,X^n \in \Rd$, and denote $\bar X := \frac{1}{n} \sum_{i=1}^n X^i$ and $\sigma^2 := \frac{1}{n} \sum_{i=1}^n \|X^i -\bar X \|^2$. Let $t \in \intset{n}$ be fixed and let $X^{\pi_1},\dots,X^{\pi_t}$ be sampled uniformly without replacement from $\{X^1,\dots,X^n\}$. Then, it holds that
\[
	\Exp[\bar X_\pi] = \bar{X} \quad \text{and} \quad \Exp[\|\bar X_\pi-\bar{X}\|^2] = \frac{n-t}{t(n-1)}\cdot\sigma^2,\quad \text{where}\quad \bar X_\pi := \frac{1}{t}{\sum}_{i=1}^t X^{\pi_i}.
	\]
\end{lemma}
Even if the set is convex in the dual space, its image in the primal space can be nonconvex. This prevents us from applying standard estimates for Bregman distances (e.g., \cite[Fact 2.5]{latafat2022bregman}) on the \emph{nonconvex} set $\cX$. In the following, we establish key estimates used throughout the analysis.

\section{Verification of DKC  --- Proof of \Cref{prop kernel}}\label{app:kernel}

\subsection{Boltzmann-Shannon entropy}
Note that the Boltzmann-Shannon entropy  $h(x)=\sum_{i=1}^d x_i\log(x_i)$ is separable with each block being the $i$-th coordinate $x_i$. We only need to verify the kernel conditioning for each  $h_i(z):=z\log(z)$. Note that for any points $z, z'>0$, their distance in the dual space is given by 
$$\rho_{h_i}(z,z^\prime) = |\log(z/z^\prime)|.$$
Since $h_i^{\prime\prime}(z) = \frac{1}{z}$, for any compact set $\cX_i \subseteq \idom{h_i}=(0,+\infty)$, the condition  number over $\cX_i$ is 
\[
\kappa_i(\cX_i) = \sup_{z,z' \in \cX_i} \max\left\{ \frac{z}{z'}, \frac{z'}{z} \right\} = \sup_{z, z' \in \cX_i} \exp\big\{|\log(z/z')|\big\} = \exp \left(\diam_{h_i}(\cX_i)\right).
\]
Consequently, for any constant $\delta>0$, any compact set $\cX_i \subseteq \idom{h_i}$ with $\diam_{h_i}(\cX_i)\leq \delta$ satisfies the kernel conditioning regularity assumption with $\kappa_i(\cX_i) \leq  \exp \left(\delta\right)$.\hfill $\Box$




\subsection{Regularized Burg's entropy}
Similar to the verification of Boltzmann-Shannon entropy, we can express Burg's entropy as
 $ h(x)= \sum_{i=1}^d - \log(x_i) + \frac{\sigma}{2}x_i^2$. Now consider $h_i(z)=-\log(z) + \frac{\sigma}{2}z^2$. Direct computation gives \[
 h_i^\prime(z) = -\frac{1}{z} + \sigma z \quad \text{and} \quad h_i^{\prime\prime}(z)=\frac{1}{z^2} + \sigma.
 \]
 Given any $\cX_i\subseteq \idom{h_i}=(0,\infty)$ with $\diam_{h_i}(\cX_i)\leq \delta$. For any points $y,z\in\cX_i$, it follows from $\sigma$-strong convexity of $h_i$ that 
 $$ \big[h_i^\prime(y)-h_i^\prime(z) \big] (y-z) \leq  \frac{\rho_{h_i}^2(y,z)}{\sigma} \leq \frac{\delta^2}{\sigma},$$
 which further implies 
 \begin{equation}
     \label{eq:verify Burg 1}
     \Big(\frac{1}{z}-\frac{1}{y}\Big)(y-z) + \sigma(y-z)^2 \leq \frac{\delta}{\sigma} \quad \Longrightarrow \quad \frac{y}{z} + \frac{z}{y} \leq 2 + \frac{\delta}{\sigma}.
 \end{equation}
 The condition number of $h_i$ on $\cX_i$ is computed by
 \[
 \kappa_i(\cX_i) = \sup_{y,z\in\cX_i} \max \bigg\{\frac{z^{-2}+\sigma}{y^{-2}+\sigma},\frac{y^{-2}+\sigma}{z^{-2}+\sigma} \bigg\}. 
 \]
For arbitrary $y,z\in\cX_i$, it holds that
\begin{align*}
    \frac{z^{-2}+\sigma}{y^{-2}+\sigma} = \frac{y^2/z^2 + \sigma y^2}{1+\sigma y^2} = 1+\frac{y^2/z^2 - 1}{1+\sigma y^2} \leq 1+\frac{y^2}{z^2} \leq \frac{\delta^2}{\sigma^2} + \frac{4\delta}{\sigma} + 5,
\end{align*}
where the last inequality is due to \eqref{eq:verify Burg 1}. Similarly, we can derive the same bound for $\frac{y^{-2}+\sigma}{z^{-2}+\sigma}$. Thus, $\kappa_i(\cX_i)\leq \frac{\delta^2}{\sigma^2} + \frac{4\delta}{\sigma} + 5$ whenever $\diam_{h_i}(\cX_i)\leq \delta$. \hfill $\Box$

\subsection{Fermi-Dirac entropy}
 To show that kernel conditioning of the second kind is valid, it suffices to demonstrate that $h_i(z):=z\log(z) + (1-z)\log(1-z)$ has bounded conditional number $\kappa_i$ for any $\cX_i\subseteq \idom{h_i} = (0,1)$ satisfying $\diam_{h_i}(\cX_i)\leq \delta$. For arbitrary points $y,z\in\cX_i$, we have 
    \begin{equation}
        \label{eq:kernel fermi 1}
        |h^\prime_i(z) - h^\prime_i(y)| \leq \delta \quad \Longleftrightarrow \quad  |\log(\tfrac{z}{1-z}) - \log(\tfrac{y}{1-y})| \leq \delta.
    \end{equation}
    Simple computation yields $h_i^{\prime\prime}(z) = \frac1x + \frac{1}{1-z} = \frac{1}{z(1-z)}$. Hence, it holds that
   \begin{equation}\label{eq:kernel fermi 2}
        \kappa_i(\cX_i)=\sup_{y,z\in\cX_i} \max\Big\{\frac{z(1-z)}{y(1-y)},\frac{y(1-y)}{z(1-z)}  \Big\} =\sup_{y,z\in\cX_i} \exp\{|\log(z(1-z)) - \log(y(1-y))|\}.
    \end{equation}
    Next, we show that $|\log(z(1-z)) - \log(y(1-y))|$ is bounded by a $\delta$-dependent constant given that \eqref{eq:kernel fermi 1} holds. We proceed by considering two main cases.\\[-1mm]

    \noindent\textbf{Case 1:} $|z-\frac12| \leq |y-\frac12|$. In this situation, $z(1-z) \geq y(1-y)$. Thus, 
    \begin{equation}\label{eq:kernel fermi 3}
        |\log(z(1-z)) - \log(y(1-y))| = \log(z(1-z)) - \log(y(1-y))= \log(\tfrac{z}{y}) + \log(\tfrac{1-z}{1-y}).
    \end{equation}
    We split this further into two sub-cases. \vspace{0.2cm}\\
    \noindent \textbf{Sub-case 1-1:} $z \leq y$. It follows that $\tfrac{y}{1-y} \geq \tfrac{z}{1-z}$. Then, \eqref{eq:kernel fermi 1} becomes
     \begin{equation}\label{eq:kernel fermi 4}
    \log(\tfrac{y}{1-y}) - \log(\tfrac{z}{1-z}) \leq \delta \quad \Longleftrightarrow \quad \log(\tfrac{1-z}{1-y}) + \log(\tfrac{y}{z})\leq  \delta,
    \end{equation}
    which, combined with \eqref{eq:kernel fermi 3}, leads to
    \[
    |\log(z(1-z)) - \log(y(1-y))| \leq 2\log(\tfrac{z}{y}) + \delta \leq \delta,
    \]
    where the last inequality is due to $\tfrac{z}{y} \leq 1$.\vspace{0.2cm}

     \noindent \textbf{Sub-case 1-2:} If $z > y$, then \eqref{eq:kernel fermi 1} can be written as
     \begin{equation}\label{eq:kernel fermi  5}
     \log(\tfrac{z}{1-z}) -  \log(\tfrac{y}{1-y}) \leq \delta \quad \Longleftrightarrow \quad \log(\tfrac{z}{y}) + \log(\tfrac{1-y}{1-z}) \leq  \delta.
     \end{equation}
     Combining this with \eqref{eq:kernel fermi 3} gives 
     $$|\log(z(1-z)) - \log(y(1-y))| \leq 2\log(\tfrac{1-z}{1-y})  + \delta.$$ 
     Since $y<z$, it holds that $\tfrac{1-z}{1-y} < 1$. Therefore,  $|\log(z(1-z)) - \log(y(1-y))| \leq \delta$.\\

     \noindent \textbf{Case 2:} $|z-\frac12| > |y-\frac12|$. In this case, $z(1-z) < y(1-y)$ and
    \begin{equation}\label{eq:kernel fermi 6}
        |\log(z(1-z)) - \log(y(1-y))| = \log(\tfrac{y}{z}) + \log(\tfrac{1-y}{1-z}).
    \end{equation}
    \noindent \textbf{Sub-case 2-1:} $z \leq y$. From \eqref{eq:kernel fermi 4} and \eqref{eq:kernel fermi 6}, we have
    \[
    |\log(z(1-z)) - \log(y(1-y))| \leq 2\log(\tfrac{1-y}{1-z}) + \delta \leq \delta.
    \]
    The last inequality is because $\tfrac{1-y}{1-z}\leq 1$.\vspace{0.2cm}

     \noindent \textbf{Sub-case 2-2:} $z > y$. Invoking the relations \eqref{eq:kernel fermi 5} and \eqref{eq:kernel fermi 6}, this results in
   \[
   |\log(z(1-z)) - \log(y(1-y))| \leq 2\log(\tfrac{y}{z}) + \delta \leq \delta.
   \]
%
   Summarizing all sub-cases, we deduce that $
   |\log(z(1-z)) - \log(y(1-y))| \leq \delta.$
    Consequently, from \eqref{eq:kernel fermi 2}, we obtain
   \[\kappa_i(\cX_i) =\sup_{y,z\in\cX_i} \exp\{|\log(z(1-z)) - \log(y(1-y))|\} \leq \exp(\delta).\]
This establishes that 
kernel conditioning of the Fermi-Dirac entropy. \hfill $\Box$

\subsection{Multi-block power kernel}
The kernel $h(x)=\frac{1}{r+2}\sum_{j=1}^m\|x\|_{[j]}^{r+2} + \frac12\|x\|^2$, $r\geq0$ has domain $\mathcal{Z}\equiv \Rd$ and can be written as \[h(x)={\sum}_{j=1}^m h_j(x_{[j]}) \quad \text{where} \quad \quad h_j(z):=\frac{1}{r+2}\|z\|^{r+2} + \frac12\|z\|^2,\]
 we thus consider the condition number of each $h_j$ over the set $\cX_j$ with bounded diameter in the sense that $\rho_{h_j}(y,z)=\|\nabla h_j(y) - \nabla h_j(z)\| \leq \delta$ for all $y,z\in\cX_j$ and for some $\delta>0$. Specifically, we want show that $\kappa_j(\cX_j)$ is upper bounded by some $\delta$-dependent constant.
Clearly, $h_j$ is $1$-strongly convex. Hence, for all $y,z\in\cX_j$ satisfying $\diam_{h_j}(\cX_j) \leq \delta$, we have 
\[
\|y-z\| \leq \rho_{h_j}(y,z) \leq \delta.
\]
The rest of proof is similar to \cite[Proposition 2.5]{zhang2024stochastic}.

Direct calculation shows $\nabla^2 h_j(z) =(\|z\|^r+1) \cdot I + r\|z\|^{r-2}\cdot zz^\top$. For all $z\in\cX_j$, it holds that $\lambda_{\min}(z)=\|z\|^r+1$ and $\lambda_{\max}(z)=(r+1)\|z\|^r+1$. Let $z\in\argmin_{u\in\cX_j}\|u\|$ and $y\in \argmax_{u\in\cX_j}\|u\|$,
\[
\kappa_{j}(\cX_j) = \frac{\lambda_{\max}(y)}{\lambda_{\min}(z)} = \frac{(r+1)\|y\|^r+1}{\|z\|^r+1}\leq \frac{(r+1)(\|z\|+\delta)^r+1}{\|z\|^r+1}.
\]
Hence, our goal boils down to finding the supremum of 
\[
\sup\bigg\{\frac{(r+1)(t+\delta)^r+1}{t^r+1}:t\geq0 \bigg\}.
\]
If $t\leq 1$, we have $\frac{(r+1)(t+\delta)^r+1}{t^r+1} \leq(r+1)(1+\delta)^r+1 $. On the other hand, when $t\geq 1$,
\begin{align*}
    \frac{(r+1)(t+\delta)^r+1}{t^r+1}  \leq \frac{(r+1)(t+\delta)^r}{t^r} = (r+1)(1+\delta/t)^r \leq (r+1)(1+\delta)^r.
\end{align*}
This completes the proof.\hfill $\Box$
 
\section{Proof of key estimates} 
\subsection{Proof of \Cref{lem:Lipschitz continuity-m}}\label{proof:lem:Lipschitz continuity-m}
\begin{proof}
Since $\cX_j$ is convex for all $j\in\intset{m}$, it follows from \cite[Lemma E.2]{zhang2024stochastic} that 
\begin{equation}\label{Lipschitz continuity-m 1}
     \frac{1}{\mu_j(\cX_j)}\cdot\|\nabla_{[j]}\; f(x) - \nabla_{[j]}\; f(y)\|^2 \leq 8\sL^2\kappa(\cX)  \sum_{s=1}^m \;\cL_s(\cX_s)\cdot\|x_{[s]}-y_{[s]}\|^2.
\end{equation}
By \Cref{lem:important bounds}, we have $\|x_{[s]}-y_{[s]}\|^2\leq \rho_{h_s}(x_{[s]},y_{[s]})^2/\mu_s(\cX_s)^2$ for all $s\in\intset{m}$, which, combined with  \eqref{Lipschitz continuity-m 1} and $\kappa_s(\cX_s) \leq \kappa(\cX)$, gives
\[
    \frac{1}{\mu_j(\cX_j)}\cdot\|\nabla_{[j]}\; f(x) - \nabla_{[j]}\; f(y)\|^2 \leq 8\sL^2\kappa(\cX)^2\cdot {\sum}_{s=1}^m \; \frac{1}{\mu_s(\cX_s)}\cdot\rho_{h_s}(x_{[s]},y_{[s]})^2.
\]
Summing the above from $j=1$ to $m$, we obtain
\begin{equation}\label{Lipschitz continuity-m 2}
{\sum}_{j=1}^m\; \frac{1}{\mu_j(\cX_j)}\cdot\|\nabla_{[j]} f(x) - \nabla_{[j]} f(y)\|^2 \leq 8m\cdot \sL^2\kappa(\cX)^2 {\sum}_{s=1}^m \; \frac{1}{\mu_s(\cX_s)}\cdot\rho_{h_s}(x_{[s]},y_{[s]})^2,
\end{equation}
which finishes the first part of the proof. 
Next, we proceed to prove that
\begin{equation}\label{Lipschitz continuity-m 2.5}
\sum_{j=1}^m \frac{\|\nabla_{[j]} f(x) - \nabla_{[j]} f(y)\|^2}{\mu_j(\cX_j)} \leq 8e(3\ln(R)+2)\cdot \sL^2\kappa(\cX)^2 \sum_{s=1}^m \frac{\rho_{h_s}(x_{[s]},y_{[s]})^2}{\mu_s(\cX_s)}
\end{equation}
with $R:=\frac{\max_j\mu_j(\cX_j)}{\min_j\mu_j(\cX_j)}$. When $R=1$, it holds that $\mu_i(\cX_i) = \mu_j(\cX_j) =\mu(\cX)$ for all $i,j\in\intset{m}$, then
\[
\sum_{j=1}^m \frac{\|\nabla_{[j]} f(x) - \nabla_{[j]} f(y)\|^2}{\mu_j(\cX_j)} = \frac{1}{\mu(\cX)} \|\nabla f(x) - \nabla f(y)\|^2.
\]
Thus, the bound \eqref{Lipschitz continuity-m 2.5} is immediate according to \Cref{lem:Lipschitz continuity}.

In what follows, we focus on $R>1$. 
Our strategy is to group those blocks that are close in terms of $\mu_j(\cX_j)$ and form a new block. 
Setting $\tau=\min\{e^{1/3},R\}$, we define indices sets $T_k\subseteq\intset{m}$ as 
\[
T_k:=\Big\{t\in\intset{m}:\tau^{k-1}\leq \frac{\mu_t(\cX_t)}{\min_j \mu_j(\cX_j)} < \tau^k \Big\}\quad \text{for all}\; 1\leq k \leq m^\prime \quad \text{where}\quad m^\prime:=\Big\lfloor \frac{\ln(R)}{\ln(\tau)} \Big\rfloor + 1
\]
and the new blocked components of $h$ as
\[
h_{T_k}(x_{[T_k]}) := {\sum}_{i\in T_k} h_i(x_{[i]}) \quad \text{for all}\; 1\leq k \leq m^\prime. 
\] 
Let $j^*\in T_k$ be such that $j^*=\argmax_{j\in T_k} \cL_j(\cX_j)$, then, by the construction of $T_k$, we have \[\kappa_{T_k}(\cX_{T_k})=\frac{\cL_{j^*}(\cX_{j^*})}{\min_{j\in T_k}\mu_{j}(\cX_j)} = \frac{\cL_{j^*}(\cX_{j^*})}{\mu_{j^*}(\cX_{j^*})}   \cdot \frac{\mu_{j^*}(\cX_{j^*})}{\min_{j\in T_k}\mu_{j}(\cX_j)} \leq \tau\cdot \kappa_j(\cX_{j^*}) \leq \tau\cdot \kappa(\cX). \] Note also that $h(x)=\sum_{k=1}^{m^\prime}h_{T_k}(x_{[T_k]})$, by mirroring the derivation of \eqref{Lipschitz continuity-m 2}, we obtain 
\begin{equation}
    \label{Lipschitz continuity-m 3}
    \sum_{j=1}^{m^\prime}\; \frac{1}{\mu_{T_j}(\cX_{T_j})}\cdot\|\nabla_{[T_j]} f(x) - \nabla_{[T_j]} f(y)\|^2 \leq 8m^\prime \tau^2 \cdot \sL^2\kappa(\cX)^2 \sum_{j=1}^{m^\prime}\; \frac{1}{\mu_{T_j}(\cX_{T_j})}\cdot\rho_{h_{T_j}}(x_{[T_j]},y_{[T_j]})^2.
\end{equation}
Moreover, it holds that 
\begin{equation}
    \label{Lipschitz continuity-m 4}
    \begin{aligned}
&\hspace{5mm}\sum_{j=1}^m \frac{1}{\mu_j(\cX_j)}\cdot\|\nabla_{[j]} f(x) - \nabla_{[j]} f(y)\|^2
 = \sum_{j=1}^{m^\prime} \sum_{s\in T_j} \frac{1}{\mu_s(\cX_s)}\cdot\|\nabla_{[s]} f(x) - \nabla_{[s]} f(y)\|^2\\ 
& \leq \sum_{j=1}^{m^\prime} \frac{1}{\min_{i\in T_j}\mu_i(\cX_i)}\sum_{s\in T_j} \|\nabla_{[s]} f(x) - \nabla_{[s]} f(y)\|^2 =\sum_{j=1}^{m^\prime} \frac{1}{\mu_{T_j}(\cX_{T_j})}\cdot\|\nabla_{[T_j]} f(x) - \nabla_{[T_j]} f(y)\|^2.
\end{aligned}
\end{equation}
Again, by the construction of $T_k$, we have
\begin{align}
    \label{Lipschitz continuity-m 5}
    \sum_{j=1}^{m^\prime} \frac{1}{\mu_{T_j}(\cX_{T_j})}\cdot\rho_{h_{T_j}}(x_{[T_j]},y_{[T_j]})^2  &\leq \sum_{j=1}^{m^\prime} \sum_{s\in T_j} \frac{\tau}{\mu_s(\cX_s)}\cdot\rho_{h_s}(x_{[s]},y_{[s]})^2\\
    & =  \tau \sum_{j=1}^m \frac{1}{\mu_j(\cX_j)}\cdot\rho_{h_j}(x_{[j]},y_{[j]})^2.\nonumber    
\end{align} 
Combining \eqref{Lipschitz continuity-m 3}--\eqref{Lipschitz continuity-m 5}, we obtain 
\[
{\sum}_{j=1}^m \;{\mu_j(\cX_j)}\cdot\|\nabla_{[j]} f(x) - \nabla_{[j]} f(y)\|^2 \leq 8m^\prime \tau^3 \cdot \sL^2\kappa(\cX)^2 {\sum}_{j=1}^m \;\frac{1}{\mu_j(\cX_j)}\cdot\rho_{h_j}(x_{[j]},y_{[j]})^2.
\]
Since $\tau=\min\{e^{1/3},R\}$ and $m^\prime \leq \ln(R)/\ln(\tau) + 1 \leq 2+ 3\ln(R)$, we have $
m^\prime \tau^3 \leq [2+3\ln(R)]\cdot e,$
which, along with the above estimate, shows \eqref{Lipschitz continuity-m 2.5}.
\end{proof}
\subsection{Proof of \cref{lem:preliminary descent}}
\label{proof:lem:preliminary descent}
It follows from the extended descent \eqref{eq:extended descent} that
\begin{align}\label{eq:first-descent-1} 
    f(x^{k+1}) - f(x^k) &\leq   \langle \nabla f(x^k) , x^{k+1} - x^k \rangle + \sL \cD(x^{k+1},x^k) \\
    &= \frac1n{\sum}_{i=1}^n \!\big \langle \nabla\! f_{\pi^k_i}(x^k) \!-\! \nabla\! f_{\pi^k_i}(y^{k,i}) , x^{k+1} - x^k \big \rangle \nonumber \\ &\quad + \frac{1}{n\alpha_k}{\sum}_{i=1}^n\! \langle  \alpha_k  \nabla\! f_{\pi^k_i}(y^{k,i}) , x^{k+1} - x^k \rangle + \sL \cD(x^{k+1},x^k).\nonumber
\end{align} 
By the update of RRMD, we have $\alpha_k \nabla f_{\pi^k_i}(y^{k,i}) = \nabla h(y^{k,i}) - \nabla h(y^{k,i+1})$. Thus,
\begin{align}
\label{eq:first-descent-1.5}
\sum_{i=1}^n \langle  \alpha_k \nabla f_{\pi^k_i}(y^{k,i}) , x^{k+1} - x^k \rangle &=  \langle \nabla h(x^k) - \nabla h(x^{k+1}) , x^{k+1} - x^k \rangle\\
&= - \cD(x^{k+1},x^k) -\cD(x^k,x^{k+1}),\nonumber
\end{align} 
where the last equation is by \Cref{lem:three point}. Note that $x^k,x^{k+1},y^{k,1},\ldots,y^{k,n} \in \cX^k$ due to \Cref{lem:region}, then we have for all $i\in\intset{n}$ that 
\begin{align}\label{eq:first-descent-2.5}
    &\hspace{0.43cm}\big \langle \nabla\! f_{\pi^k_i}(x^k) \!-\! \nabla f_{\pi^k_i}(y^{k,i}), x^{k+1} \!-\! x^k  \big\rangle \nonumber\\[1mm]
    &\leq \frac{n\alpha_k\kappa_{\delta} }{2\mu(\cX^k)} \big\|\nabla\! f_{\pi^k_i}(x^k) \!-\! \nabla\! f_{\pi^k_i}(y^{k,i}) \big\|^2 \!+\! \frac{\mu(\cX^k)\|x^{k+1}\!-\!x^k\|^2}{2n\alpha_k\kappa_{\delta}}   \\[1mm] 
    &\overset{\text{(i)}}{\leq} \frac{n\alpha_k\kappa_{\delta}  }{2\mu(\cX^k)} \big\| \nabla\! f_{\pi^k_i}(x^k) \!-\! \nabla f_{\pi^k_i}(y^{k,i}) \big\|^2 \!\!+\! \frac{\cD(x^k\!,x^{k+1}) \!+\! \cD(x^{k+1}\!,x^k)}{2n\alpha_k}, \nonumber
\end{align}
where (i) uses $\mu(\cX^k)\|x^k-x^{k+1}\|^2\leq \kappa_{\delta}\cdot [\cD(x^{k+1},x^k)+\cD(x^k,x^{k+1})]$ inferred from \Cref{lem:important bounds}. Combining inequalities \eqref{eq:first-descent-1}--\eqref{eq:first-descent-2.5} and the fact that $\sL \leq 1/(2n\alpha_k)$ proves this lemma. \hfill $\Box$

\subsection{Proof of \Cref{lem:valid complexity}}
\label{proof of lem:valid complexity}
By the update formula \eqref{eqn:rrmd-update} of $\{x^k\}$ and the definition \eqref{eq:station measure} of $\hat x^k$, 
\begin{equation}\label{eq:update proxy}
\nabla h(x^{k+1}) + \alpha_k{\sum}_{i=1}^n \nabla f_{\pi_i^k}(y^{k,i}) = \nabla h(x^k) = \nabla h(\hat x^{k}) + \alpha_k{\sum}_{i=1}^n \nabla f_{\pi_i^k}(x^k).
\end{equation}
To quantify the region of the surrogate iterate $\hat x^k$, we first show that $\hat x^k \in \cX^k$. Recall that $\|\nabla f_i(x)\|\leq \sG$ for all $x\in\cZ$ and $i\in\intset{n}$, we have $\rho_h(\hat x^{k},x^k) \leq n\alpha_k \sG \leq \delta/2$, where the last inequality uses $\alpha_k \leq \delta/(2n\sG)$. By definition of $\cX^k$, it follows $\hat x^k \in \cX^k$.

Since $(n\alpha_k)^{2} \cdot  \cG(x^k) = \cD(x^k,\hat x^k)$, we concentrate on the upper bound of $\cD(x^k,\hat x^k)$. 
Utilizing three points identity (\Cref{lem:three point}), we have
\begin{align*}
&\hspace{5mm}\cD(x^k,\hat x^k) - \cD(x^k, x^{k+1}) - \cD(x^{k+1},\hat x^k) 
=   \langle x^k-x^{k+1},\nabla h(x^{k+1})- \nabla h(\hat x^k) \rangle\\[2mm]
& 
\overset{\text{(i)}}{\leq}  \frac{\cL(\cX^k) }{2\kappa_{\delta}} \|x^k-x^{k+1}\|^2 + \frac{\kappa_{\delta}\cdot \rho_h^2(x^{k+1},\hat x^k)}{2\cL(\cX^k)} \overset{\text{(ii)}}{\leq}  \kappa_{\delta} \cdot \cD(x^k, x^{k+1}) + \kappa_{\delta}\cdot \cD(x^{k+1}, \hat x^k),
\end{align*}    
where (i) uses $2\iprod{a}{b} \leq \cL(\cX^k)  \|a\|^2/\kappa_{\delta} + \kappa_{\delta}\|b\|^2/\cL(\cX^k)$, and (ii) follows from $\|x^k-x^{k+1}\|^2 \leq 2\kappa_{\delta}\cD(x^k,x^{k+1})/\mu(\cX^k)$ and $\rho_h^2(x^{k+1},\hat x^k)\leq 2\cL(\cX^k) \cdot \cD(x^{k+1}, \hat x^k)$ due to $x^k,x^{k+1},\hat x^k\in\cX^k$ and \Cref{lem:important bounds}. Rearranging the above estimate and noting $\kappa_{\delta}\geq 1$, we deduce 
\begin{equation}
    \label{eq:bregman bound 1}
    \cD(x^k, \hat x^{k}) \leq 2\kappa_{\delta} \cdot \cD(x^k, x^{k+1}) + 2\kappa_{\delta} \cdot\cD(x^{k+1}, \hat x^k).
\end{equation}
On the other hand, it follows from $2\mu(\cX^k) \cD(x^{k+1},\hat x^k) \leq \|\nabla h(x^{k+1}) - \nabla h(\hat x^k)\|^2$ and \eqref{eq:update proxy} that
\[\cD(x^{k+1}, \hat x^k) \leq \frac{n\alpha_k^2}{2\mu(\cX^k)} \cdot {\sum}_{i=1}^n\;  \|\nabla f_{\pi_i^k}(y^{k,i}) - \nabla f_{\pi_i^k}(x^k)\|^2 =\frac{n\alpha_k \cdot \err_k}{\kappa_{\delta}}.\]
Merging this bound into \eqref{eq:bregman bound 1} and substituting $\cG(x^k) =   \cD(x^k,\hat x^k)/(n\alpha_k)^{2}$, we obtain
\[
\cG(x^k) = \frac{\cD(x^k,\hat x^k)}{(n\alpha_k)^{2}} \leq 2\kappa(\cX^k) \cdot \frac{\cD(x^k, x^{k+1})}{(n\alpha_k)^{2}} + \frac{2\err_k}{n\alpha_k}.
\]
Note also that $(n\alpha_k)^2\|\nabla f(x^k)\|^2 = \|\nabla h(x^k) - \nabla h(\hat x^k)\|^2 \leq 2\cL(\cX^k) \cdot \cD(x^k,\hat x^k)$, then
\[
 \frac{\|\nabla f(x^k)\|^2}{\mu(\cX^k)} \leq \frac{2\kappa(\cX^k) }{(n\alpha_k)^2}\cdot \cD(x^k,\hat x^k) = 2\kappa(\cX^k) \cdot \cG(x^k).
\]
Together with $\kappa(\cX^k)\leq \kappa_\delta$, we complete the proof. \hfill $\Box$

\subsection{Proof of \Cref{lem:err bound}}\label{proof:lem:err bound} 
\noindent\textbf{Part I:} For arbitrary permutations $\{\pi^k\}_k$, as $x^k,y^{k,i}\in\cX^k$, \cref{lem:Lipschitz continuity,lem:region} indicate
\begin{align}\label{eq:err-est-1}  
\|\nabla f_{\pi_{i}^k}(y^{k,i}) - \nabla f_{\pi_{i}^k}(x^k) \|^2 & \leq  \kappa_\delta \sL^2  \cdot \rho_h(y^{k,i},x^k)^2 = \kappa_\delta \sL^2 \cdot \alpha_k^2   \Big\|{\sum}_{j=1}^{i-1} \nabla f_{\pi_j^k}(y^{k,j}) \Big\|^2\\
& \leq  \kappa_\delta \sL^2 \cdot \alpha_k^2   n{\sum}_{j=1}^{n}  \|\nabla f_{\pi_j^k}(y^{k,j})\|^2 \leq   n^2 \alpha_k^2  \sL^2 G^2 \kappa_\delta. \nonumber
\end{align}
By $\mu$-strong convexity of $h$, we prove
$\err_k \leq \frac{\kappa_\delta\cdot \alpha_k}{2\mu}{\sum}_{i=1}^n \| \nabla f_{\pi^k_i}(x^{k}) - \nabla f_{\pi^k_i}(y^{k,i}) \|^2 \leq \frac{\kappa_\delta^2 \sL^2 \sG^2}{2\mu} \cdot n^3 \alpha_k^3$.
    
\noindent\textbf{Part II:} Using the notation $\Exp_k[\cdot]:=\Exp[\cdot\mid\mathcal{F}_k]$, for permutation $\{\pi^k\}$ generated via uniform shuffling scheme, taking conditional expectation of the first row of \eqref{eq:err-est-1} gives 
\begin{equation}\label{eq:exp est 1}
     \begin{aligned}
     &\hspace{5mm}\Exp_k\big[\|\nabla f_{\pi_{i}^k}(x^k) - \nabla f_{\pi_{i}^k}(y^{k,i})\|^2 \big] \leq \kappa_\delta \sL^2\alpha_k^2 \cdot \Exp_k\Big[ \big\|{\sum}_{j=1}^{i-1} \nabla f_{\pi_j^k}(y^{k,j}) \big\|^2 \Big]\\[2mm]
     &= \kappa_\delta \sL^2 \alpha_k^2 \cdot  \Exp_k\Big[\big\|{\sum}_{j=1}^{i-1}
     \big[ \nabla f_{\pi_j^k}(y^{k,j}) - \nabla f_{\pi_j^k}(x^k) + \nabla f_{\pi_j^k}(x^k) - \nabla f(x^k)+ \nabla f(x^k) \big]\big\|^2\Big]\\[2mm]
     &\leq 2\kappa_\delta \sL^2\alpha_k^2 \cdot i\;{\sum}_{j=1}^n\Exp_k\big[\|\nabla f_{\pi_j^k}(y^{k,j}) - \nabla f_{\pi_j^k}(x^k) \|^2\big] + 4\kappa_\delta \sL^2\alpha_k^2 \cdot  i^2\|\nabla f(x^k)\|^2\\& \quad + 4\kappa_\delta \sL^2\alpha_k^2\cdot  \Exp_k\Big[\big\|{\sum}_{j=1}^{i-1} \nabla f_{\pi_j^k}(x^k) - \nabla f(x^k)\big\|^2 \Big] ,
 \end{aligned}
\end{equation}
where the last line simply utilizes $(a+b+c)^2\leq 2a^2+4b^2+4c^2$.
Leveraging the variance bound of sampling without replacement (\Cref{lem:sample-without-re}), we have 
\begin{equation}\label{eq:exp est 1.5}
\begin{aligned}
&\hspace{5mm} \Exp_k\Big[\big\|{\sum}_{j=1}^{i-1} \nabla f_{\pi_j^k}(x^k) - \nabla f(x^k)\big\|^2 \Big]  \\[2mm]
&= \frac{(i-1)(n-i+1)}{n-1} \cdot \frac1n\;{\sum}_{j=1}^n\|\nabla f_j(x^k) - \nabla f(x^k)\|^2 \\[2mm]
&\leq {\sum}_{j=1}^n\|\nabla f_j(x^k) - \nabla f(x^k)\|^2\\[2mm]
&={\sum}_{j=1}^n \|\nabla f_j(x^k)\|^2 - 2\Big\langle {\sum}_{j=1}^n \nabla f_j(x^k),\nabla f(x^k) \Big\rangle +n \|\nabla f(x^k)\|^2 \leq {\sum}_{j=1}^n \|\nabla f_j(x^k)\|^2, 
\end{aligned}
\end{equation}
where the last inequality is by $ {\sum}_{j=1}^n \!\nabla f_j(x) \!=\! n\nabla f(x)$. 
%
%
Inserting \eqref{eq:exp est 1.5} to \eqref{eq:exp est 1}, 
telescoping the resulting estimate for $i=1,\cdots,n$, and using $\sum_{i=1}^{n-1} i \leq n^2/2$, $\sum_{i=1}^{n-1}i^2 = n(n-1)(2n-1)/6 \leq n^3/2$ gives
\begin{align*}
    {\sum}_{i=1}^n\, \Exp_k\Big[\|\nabla f_{\pi_{i}^k}(x^k) - \nabla f_{\pi_{i}^k}(y^{k,i})\|^2 \Big]  
    &\leq \frac{2n\kappa_\delta \sL^2\alpha_k^2}{1 - \kappa_\delta \sL^2n^2\alpha_k^2}  \Big( n^2 \|\nabla f(x^k)\|^2 + 2 {\sum}_{j=1}^n \|\nabla f_j(x^k)\|^2\Big)\\
    &\leq 4n\kappa_\delta \sL^2\alpha_k^2\Big( n^2 \|\nabla f(x^k)\|^2 + 2 {\sum}_{j=1}^n \|\nabla f_j(x^k)\|^2\Big),
\end{align*} 
where the last line is by $\kappa_\delta \sL^2 n^2\alpha_k^2 \leq 1/2$.
In view of the definition of $\err_k$, we have 
\begin{align*}
\Exp_k[\err_k]&=\frac{\kappa_\delta \alpha_k}{2\mu(\cX^k)}{\sum}_{i=1}^n \Exp_k\Big[\| \nabla f_{\pi^k_i}(x^{k}) - \nabla f_{\pi^k_i}(y^{k,i}) \|^2 \Big]\\[1mm]
&\leq  2n\kappa_\delta^2 \sL^2\alpha_k^3\Big( n^2  \|\nabla f(x^k)\|^2/\mu(\cX^k) + [2/\mu(\cX^k)] {\sum}_{j=1}^n \|\nabla f_j(x^k)\|^2\Big)\\[1mm]
&\leq 4n\kappa_\delta^2\sL^2\alpha_k^3\Big( n^2\kappa_\delta \cdot \cG(x^k) + [\mu(\cX^k)]^{-1} {\sum}_{j=1}^n \|\nabla f_j(x^k)\|^2\Big),
\end{align*}
where the last line follows from $\|\nabla f(x^k)\|^2/\mu(\cX^k) \leq 2\kappa_\delta \cdot \cG(x^k)$ (see \Cref{lem:valid complexity}). 
The rest of proof is to substitute $[\mu(\cX^k)]^{-1}{\sum}_{j=1}^n \|\nabla f_j(x^k)\|^2 \leq n\sG^2/\mu$ into the above estimate.\hfill $\Box$

\subsection{Proof of \Cref{lem:err bound multi}}\label{proof:lem:err bound multi}
\begin{proof}
    Using \Cref{lem:Lipschitz continuity-m} and the fact $\kappa(\cX^k)\leq \kappa_\delta$ by \Cref{lem:region}, we obtain for all $i\in\intset{n}$ that
\begin{equation}\label{eq:multi block 1}
\begin{aligned}
{\sum}_{j=1}^m\frac{\|\nabla_{[j]} \;f_{\pi^k_i}(x^k) - \nabla_{[j]} f_{\pi^k_i}(y^{k,i})\|^2}{\mu_j(\cX^k_j)} &\leq \Gamma^2(\cX^k)\sL^2 \kappa_\delta^2{\sum}_{j=1}^m \frac{\rho_{h_j}^2(x^k_{[j]},y^{k,i}_{[j]})}{\mu_j(\cX^k_j)} \\ &\overset{\text{(i)}}{=} \Gamma^2(\cX^k)\sL^2 \kappa_\delta^2\alpha_k^2{\sum}_{j=1}^m \frac{1}{\mu_j(\cX^k_j)}\Big\| {\sum}_{t=1}^{i-1} \, \nabla_{[j]} \;f_{\pi^k_t}(y^{k,t})\Big\|^2,
\end{aligned}
\end{equation}
where (i) holds due to the update rule of RRMD. Utilizing the $\mu$-strong convexity of $h$ (hence $\mu_j(\cX_j^k) \geq \mu$) and the relation ${\sum}_{j=1}^m\|x_{[j]}\|^2=\|x\|^2$, we obtain
\[
{\sum}_{j=1}^m\frac{\|\nabla_{[j]} \;f_{\pi^k_i}(x^k) - \nabla_{[j]} \; f_{\pi^k_i}(y^{k,i})\|^2}{\mu_j(\cX^k_j)} \leq  \frac{\Gamma^2(\cX^k)\sL^2 \kappa_\delta^2 \alpha_k^2}{\mu} \cdot \Big\| {\sum}_{t=1}^{i-1} \, \nabla f_{\pi^k_t}(y^{k,t})\Big\|^2.
\]
By $\|\nabla f_i(x)\|\leq\sG$ for all $x\in\cZ$ and $i\in\intset{n}$, we show (a). By adding and subtracting, we rewrite \eqref{eq:multi block 1} as
\begin{align*}
&\hspace{3.5mm}\sum_{j=1}^m\frac{1}{\mu_j(\cX^k_j)} \|\nabla_{[j]} \;f_{\pi^k_i}(x^k) - \nabla_{[j]} \; f_{\pi^k_i}(y^{k,i})\|^2 \\
& \leq \Gamma^2(\cX^k)\sL^2 \kappa_\delta^2\alpha_k^2\sum_{j=1}^m \frac{1}{\mu_j(\cX^k_j)}\Big\| {\sum}_{t=1}^{i-1} \, \nabla_{[j]} \;f_{\pi^k_t}(y^{k,t})\Big\|^2 \\
&= \sum_{j=1}^m \frac{\Gamma^2(\cX^k)\sL^2 \kappa_\delta^2\alpha_k^2}{\mu_j(\cX^k_j)}\Big\| {\sum}_{t=1}^{i-1} \big[\, \nabla_{[j]} \;f_{\pi^k_t}(y^{k,t}) - \, \nabla_{[j]} \;f_{\pi^k_t}(x^k) + \, \nabla_{[j]} \;f_{\pi^k_t}(x^k) - \nabla_{[j]} \;f(x^k) + \nabla_{[j]} \;f(x^k) \big]\Big\|^2 \\
&\overset{\text{(i)}}{\leq} \Gamma^2(\cX^k)\sL^2 \kappa_\delta^2\alpha_k^2\sum_{j=1}^m \frac{1}{\mu_j(\cX^k_j)} \cdot \Big(2\Big\| {\sum}_{t=1}^{i-1} \, \nabla_{[j]} \;f_{\pi^k_t}(y^{k,t}) - \, \nabla_{[j]} \;f_{\pi^k_t}(x^k)\Big\|^2+ 4i^2\big\| \nabla_{[j]} \;f(x^k) \big\|^2\Big) \\ &\quad + \frac{4\Gamma^2(\cX^k)\sL^2 \kappa_\delta^2\alpha_k^2}{\mu} \Big\| {\sum}_{t=1}^{i-1} \, \nabla f_{\pi^k_t}(x^k) - \nabla f(x^k)\Big\|^2  \\
&\leq 2\Gamma^2(\cX^k) n\sL^2 \kappa_\delta^2\alpha_k^2 \sum_{t=1}^{n}  \sum_{j=1}^m \frac{1}{\mu_j(\cX^k_j)} \|\nabla_{[j]} \;f_{\pi^k_t}(y^{k,t}) - \, \nabla_{[j]} \;f_{\pi^k_t}(x^k)\|^2 \\ &\quad + 4i^2\cdot \Gamma^2(\cX^k) \sL^2 \kappa_\delta^2\alpha_k^2 \sum_{j=1}^m \frac{\big\| \nabla_{[j]} \;f(x^k) \big\|^2}{\mu_j(\cX^k_j)}  + \frac{4\Gamma^2(\cX^k)\sL^2 \kappa_\delta^2\alpha_k^2}{\mu}\cdot \Big\| {\sum}_{t=1}^{i-1} \, \nabla f_{\pi^k_t}(x^k) - \nabla f(x^k)\Big\|^2,
\end{align*}
where (i) uses $(a+b+c)^2 \leq 2a^2 + 4b^2 + 4c^2$, $\mu_j(\cX^k_j)\geq\mu$, and ${\sum}_{j=1}^m\|x_{[j]}\|^2=\|x\|^2$. Summing the above estimate over $i=1,\ldots,n$, rearranging, and using $\sum_{i=1}^{n-1}i^2\leq n^3/2$, then
\begin{equation}\label{eq:multi block 2}
    \begin{aligned}
    &\hspace{5mm}\big[1-2\Gamma^2(\cX^k)\sL^2 \kappa_\delta^2n^2\alpha_k^2\big]   {\sum}_{i=1}^{n}  {\sum}_{j=1}^m \frac{1}{\mu_j(\cX^k_j)} \|\nabla_{[j]} \;f_{\pi^k_i}(y^{k,i}) - \, \nabla_{[j]} \;f_{\pi^k_i}(x^k)\|^2 \\&\leq 4\Gamma^2(\cX^k)\sL^2 \kappa_\delta^2 \alpha_k^2\cdot \Big(\mu^{-1}{\sum}_{i=1}^n \Big\| {\sum}_{t=1}^{i-1} \, \nabla f_{\pi^k_t}(x^k) - \nabla f(x^k)\Big\|^2 + \frac{n^3}{2} {\sum}_{j=1}^m\frac{\big\| \nabla_{[j]} \;f(x^k) \big\|^2}{\mu_j(\cX^k_j)} \Big).
\end{aligned}
\end{equation}
Next, we take conditional expectation $\Exp_k[\cdot]$ of \eqref{eq:multi block 2} and note that $[1-2\Gamma^2(\cX^k)\sL^2 \kappa_\delta^2n^2\alpha_k^2]  \geq \frac12$, then
\begin{equation}\label{eq:multi block 3}
   \begin{aligned}
    &\hspace{5mm}  {\sum}_{i=1}^{n}  {\sum}_{j=1}^m \frac{1}{\mu_j(\cX^k_j)} \cdot \Exp_k\big[\|\nabla_{[j]} \;f_{\pi^k_i}(y^{k,i}) - \, \nabla_{[j]} \;f_{\pi^k_i}(x^k)\|^2\big] \\&\leq 8\Gamma^2(\cX^k)\sL^2 \kappa_\delta^2 \alpha_k^2 \cdot \Big(\mu^{-1}{\sum}_{i=1}^n \Exp_k\Big[\Big\| {\sum}_{t=1}^{i-1} \, \nabla f_{\pi^k_t}(x^k) - \nabla f(x^k)\Big\|^2 \Big]+ \frac{n^3}{2} {\sum}_{j=1}^m\frac{\big\| \nabla_{[j]} \;f(x^k) \big\|^2}{\mu_j(\cX^k_j)}\Big)\\
    &\leq \frac{8\Gamma^2(\cX^k)\kappa_\delta^2\sL^2\sG^2}{\mu} \cdot n^2\alpha_k^2 + 4\Gamma^2(\cX^k)\sL^2 \kappa_\delta^2 n^3\alpha_k^2 \;{\sum}_{j=1}^m\frac{\big\| \nabla_{[j]} \;f(x^k) \big\|^2}{\mu_j(\cX^k_j)},
\end{aligned} 
\end{equation}
where the last line uses \eqref{eq:exp est 1.5} and $\|\nabla f_i(x)\|\leq \sG$. 
According to the definition of stationarity measure $\cG(x^k)$ and $\hat x^k$ in \eqref{eq:station measure}, it holds that
\[
{\sum}_{j=1}^m\frac{\big\| \nabla_{[j]} \;f(x^k) \big\|^2}{\mu_j(\cX^k_j)} = {\sum}_{j=1}^m\frac{\rho_{h_j}(x^k_{[j]},\hat x^k_{[j]})^2}{\mu_j(\cX^k_j)\cdot (n\alpha_k)^2} \leq \frac{2\kappa_\delta}{(n\alpha_k)^2} {\sum}_{j=1}^m{\cal D}_{h_j}(x^k_{[j]},\hat x^k_{[j]}) = 2\kappa_\delta \cdot \cG(x^k),
\]
where the inequality follows from 
\Cref{lem:important bounds}. Inserting above relation into \eqref{eq:multi block 3} and recalling that 
\[
\Exp_k[\cE_k]=\frac{\kappa_\delta \cdot \alpha_k}{2} {\sum}_{i=1}^n {\sum}_{j=1}^m \frac{1}{\mu_j(\cX^k_j)} \cdot \Exp_k[\| \nabla_{[j]}\, f_{\pi^k_i}(x^{k}) - \nabla_{[j]}\, f_{\pi^k_i}(y^{k,i}) \|^2],
\]
we thus finalize the proof.
\end{proof}

\bibliographystyle{siam}
\bibliography{references}

\end{document}